\documentclass[12pt]{amsart}
\usepackage{amsmath}
\usepackage{amstext}
\usepackage{amsfonts}
\usepackage{amssymb}
\usepackage{amsthm}
\usepackage{amsrefs}
\usepackage{verbatim}
\usepackage{tikz, tikz-cd}
\usetikzlibrary{matrix, arrows, positioning}
\usepackage{ifthen}
\usepackage{float}

\allowdisplaybreaks

\usepackage{microtype}
\usepackage{color,hyperref}
\definecolor{darkblue}{rgb}{0.0,0.0,0.3}
\hypersetup{colorlinks,breaklinks,linkcolor=red,urlcolor=red,anchorcolor=red,citecolor=red}

\usepackage{mathtools}
\mathtoolsset{centercolon}
\usepackage{enumitem}

\theoremstyle{plain}
\newtheorem{thm}{Theorem}[section]		
\newtheorem*{thm*}{Theorem}
		
\newtheorem{cor}[thm]{Corollary}
\newtheorem*{cor*}{Corollary}
\newtheorem{prop}[thm]{Proposition}		
\newtheorem{lem}[thm]{Lemma}		

\newtheorem{introtheorem}{Theorem}

\newtheorem{introcorollary}[introtheorem]{Corollary}

\theoremstyle{definition}
\newtheorem{defn}[thm]{Definition}
\newtheorem{example}[thm]{Example}
\newtheorem{rem}[thm]{Remark}

\numberwithin{equation}{section}

\newcommand{\R}{{\mathcal{R}}}
\renewcommand{\S}{{\mathcal{S}}}

\newcommand{\C}{{\mathcal{C}}}

\newcommand{\bC}{{\mathbb{C}}}
\newcommand{\bF}{{\mathbb{F}}}
\newcommand{\bN}{{\mathbb{N}}}
\newcommand{\bR}{{\mathbb{R}}}

\newcommand{\bZ}{{\mathbb{Z}}}

\newcommand{\cF}{\mathcal{F}}

\newcommand{\fB}{{\mathfrak{B}}}
\newcommand{\fC}{{\mathfrak{C}}}

\newcommand{\rA}{{\mathrm{A}}}

\newcommand{\rC}{{\mathrm{C}}}

\newcommand{\rP}{{\mathrm{P}}}

\renewcommand{\phi}{\varphi}

\DeclareMathOperator{\id}{\operatorname{id}}

\DeclareMathOperator{\co}{\mathrm{conv}}

\DeclareMathOperator{\Obj}{Obj}
\DeclareMathOperator{\Mor}{Mor}

\newcommand{\ol}{\overline}

\DeclareMathOperator{\fb}{\partial_{\mathrm{sp}}}
\DeclareMathOperator{\gb}{\partial_{\mathrm{p}}}
\DeclareMathOperator{\mb}{\partial_{\mathrm{m}}}

\begin{document}
\title[Amenability, proximality and higher order syndeticity
]{Amenability, proximality and higher order syndeticity}

\author[M. Kennedy]{Matthew Kennedy}
\address{Department of Pure Mathematics\\ University of Waterloo\\Waterloo, ON, N2L 3G1, Canada}
\email{matt.kennedy@uwaterloo.ca}
\author[S. Raum]{Sven Raum}
\address{Department of Pure Mathematics \\ University of Stockholm \\ 106 91 Stockholm \\ Sweden and 
  Institute of Mathematics of the Polish Academy of Sciences \\ ul.\ \'Sniadeckich 8 \\ 00-656 Warszawa \\ Poland}
\email{raum@math.su.se}
\author[G. Salomon]{Guy Salomon}
\address{Department of Mathematics\\ Weizmann Institute of Science\\Rehovot 7610001, Israel}
\email{guy.salomon@weizmann.ac.il}

\begin{abstract}
We show that the universal minimal proximal flow and the universal minimal strongly proximal flow of a discrete group can be realized as the Stone spaces of translation invariant Boolean algebras of subsets of the group satisfying a higher order notion of syndeticity. We establish algebraic, combinatorial and topological dynamical characterizations of these subsets that we use to obtain new necessary and sufficient conditions for strong amenability and amenability. We also characterize dense orbit sets, answering a question of Glasner, Tsankov, Weiss and Zucker.
\end{abstract}

\subjclass[2010]{}
\keywords{}
\thanks{First author supported by the Canadian Natural Sciences and Engineering Research Council grant number 50503-10787. Second author supported by the Swedish Research Council grant number 2018-04243 and the European Research Council under the European Union's Horizon 2020 research and innovation programme grant number 677120-INDEX. Third author supported by a Rothschild Fellowship and a Koshland Fellowship.}

\maketitle

\setcounter{tocdepth}{2}
\tableofcontents


\section{Introduction}

Recently, several breakthrough results have been obtained that clarify the relationship between certain important algebraic and dynamical properties of discrete groups. Among the most prominent are the characterization of strongly amenable groups by Frisch, Tamuz and Vahidi Ferdowsi \cite{FTV2019} (see also \cite{GTWZ2019}), and the characterization of Choquet-Deny groups by Frisch, Hartman, Tamuz and Vahidi Ferdowsi \cite{frischhartmantamuzferdowsi18} (see also \cite{EK2019}).

In this paper, we continue this line of research from a slightly different perspective. A key technical idea is the use of a device that we view as a kind of ``topological Furstenberg correspondence'' that allows us to give algebraic descriptions of the universal minimal proximal flow and the universal minimal strongly proximal flow associated to a group.

The descriptions are in terms of subsets of the group that we call completely syndetic and strongly complete syndetic respectively, because they are higher order variants of syndetic subsets. These subsets have very interesting algebraic, combinatorial and topological dynamical properties. In addition, it turns out that the existence of disjoint subsets with these properties completely characterizes both strong amenability and amenability.

Specifically, for a discrete group $G$ and $n \in \bN$, we say that a subset $A \subseteq G$ is {\em completely syndetic} if finitely many translates of the Cartesian product $A^n$ by elements in $G$ cover $G^n$, where $G$ acts diagonally on $G^n$. Equivalently and more succinctly, if there is a finite subset $F \subseteq G$ such that $FA^n = G^n$. The definition of a {\em strongly completely syndetic} subset is similar, and every strongly completely syndetic subset is, as the name suggests, completely syndetic.

The next result, which is Theorem \ref{thm:minimal-proximal-completely-syndetic}, realizes the universal minimal proximal flow $\gb G$ as the Stone space of a translation invariant Boolean algebra of completely syndetic subsets of $G$. This is an analogue of a result of Balcar and Franek \cite{BF1997} that realizes the universal minimal $G$-flow $\mb G$ as the Stone space of a translation invariant Boolean algebra of syndetic subsets of $G$.

\begin{introtheorem}
Let $G$ be a discrete group. Then the universal minimal proximal $G$-flow $\gb G$ is isomorphic to the Stone space of any translation invariant Boolean subalgebra of $2^G$ that is maximal with respect to the property that every non-empty element is completely syndetic.
\end{introtheorem}

The main result in \cite{FTV2019} is that a discrete group is strongly amenable if and only if it is FC-hypercentral, meaning that it has no non-trivial ICC quotient, i.e. no non-trivial quotient with the property that every non-trivial conjugacy class is infinite. We show directly in Section \ref{thm:strong-amenability} that the main technical step in their paper, which consists of constructing a family of subshifts satisfying certain disjointedness conditions, is equivalent to the construction of completely syndetic subsets. 

Since the strong amenability of $G$ is equivalent to the triviality of $\gb G$, as a consequence of this observation and the previous theorem, we obtain, in Theorem \ref{thm:strong-amenability}, a new characterization of strong amenability (and hence of FC-hypercentrality).

\begin{introcorollary}
Let $G$ be a discrete group. Then $G$ is not strongly amenable if and only if there is a proper normal subgroup $H \unlhd G$ such that for every finite subset $F \subseteq G \setminus H$, there is a completely syndetic subset $A \subseteq G$ such that $FA \cap A = \emptyset$.
\end{introcorollary}

The next result, which is Theorem \ref{thm:minimal-strongly-proximal-completely-syndetic}, is a similar realization of the universal minimal strongly proximal flow $\fb G$ as the Stone space of a Boolean algebra of strongly completely syndetic subsets of $G$.

\begin{introtheorem}
Let $G$ be a discrete group. Then the universal minimal strongly proximal $G$-flow $\fb G$ is isomorphic to the Stone space of any translation invariant Boolean subalgebra of $2^G$ that is maximal with respect to the property that every non-empty element is strongly completely syndetic.
\end{introtheorem}

Since the amenability of $G$ is equivalent to the triviality of $\fb G$, we also obtain, in Theorem \ref{thm:amenability-test}, a new characterization of amenability that seems to have a different flavor than existing characterizations in terms of paradoxical subsets and F\o{}lner sequences.

\begin{introcorollary}
Let $G$ be a discrete group. Then $G$ is not amenable if and only if there is a subset $A \subseteq G$ such that both $A$ and $G \setminus A$ are strongly completely syndetic.
\end{introcorollary}

A key idea in our paper is the correspondence between between totally disconnected $G$-flows and translation invariant Boolean subalgebras of $2^G$ that we view as a kind of ``topological Furstenberg correspondence.'' Zucker \cite{Zuc2020} showed that the $G$-flows $\gb G$ and $\fb G$ are maximally highly proximal, implying that they are extremally disconnected (and therefore totally disconnected).

An important observation in \cite{KK2017}, is that the C*-algebra $\rC(\fb G)$ of continuous functions on $\fb G$ is injective in a certain strong sense. One consequence is the existence of relatively invariant measures in the sense of \cite{Gla1975}, which proves to be particularly useful when used in combination with the above-mentioned correspondence.

More generally, we show that for a $G$-flow $X$, the C*-algebra $\rC(X)$ is injective in the above sense if and only if $X$ is what we call {\em maximally affinely highly proximal}, which is a strengthening of the property of being maximally highly proximal. We prove in Theorem \ref{thm:boundaries-injective} that in addition to the $G$-flow $\fb G$, the $G$-flows $\mb G$ and $\gb G$ also have this property.

\begin{introtheorem}
 Let $G$ be a discrete group. Then the Stone--\v{C}ech compactification $\beta G$, the universal minimal $G$-flow $\mb G$, the universal minimal proximal $G$-flow $\gb G$ and the universal strongly proximal $G$-flow $\fb G$ are all maximally affinely highly proximal.
\end{introtheorem}

The perspective we take in this paper allows us to answer a question of Glasner et al. from \cite{GTWZ2019}. They asked for a characterization of dense orbit sets, which are subsets $A \subseteq G$ with the property that for every minimal $G$-flow $X$ and every point $x \in X$, the set $Ax$ is dense in $X$. The following characterization of dense orbit sets, which is implied by Theorem \ref{thm:dense-orbit-set-dual} and the results in Section \ref{sec:dense-orbit-sets}, provides an answer to this question in terms of syndetic sets.

\begin{introtheorem}
Let $G$ be a discrete group. A subset $A \subseteq G$ is a dense orbit set if and only if there is no subset $B \subseteq A^c$ with the property that for every pair of finite subsets $F_1 \subseteq B$ and $F_2 \subseteq B^c$, the set $(\cap_{f_1 \in F_1} f_1^{-1} B) \cap (\cap_{f_2 \in F_2} f_2^{-1} B^c)$ is syndetic. 
\end{introtheorem}

Other results in Section \ref{sec:dense-orbit-sets} may also be of interest. For example, we establish a characterization of dense orbit sets in terms of the semigroup structure of the Stone--\v{C}ech compactification of $G$.

\section*{Acknowledgements}

The authors are grateful to Eli Glasner, Nicholas Manor, Todor Tsankov and Andy Zucker for pointing out errors in and providing helpful feedback on a previous draft of this paper.

\section{Preliminaries}

In this section we will briefly review some of the basic facts that we will require from the theory of topological dynamics and the theory of C*-algebras. For a reference on topological dynamics, we direct the reader to Glasner's monograph \cite{Gla1976}. For a reference on C*-algebras, we direct the reader to Arveson's monograph \cite{Arv1976}.

\subsection{Flows}

Let $G$ be a discrete group. A compact Hausdorff space $X$ is a $G$-flow if there is an action of $G$ on $X$ by homeomorphisms. For $g \in G$ and $x \in X$, we will write $gx$ for the image of $x$ under the homeomorphism corresponding to $g$. We will consider the category of $G$-flows with equivariant continuous maps as morphisms.

A $G$-flow $Y$ is an {\em extension} of $X$ and $X$ is a {\em factor} of $Y$ if there is an equivariant continuous surjective map $\alpha : Y \to X$. We will refer to the map $\alpha$ as an {\em extension}.

Almost all of the flows that we will consider in this paper will be minimal. A $G$-flow $X$ is {\em minimal} if for every point $x \in X$, the orbit $Gx = \{gx : g \in G\}$ is dense. There is a unique minimal $G$-flow $\mb G$ that is universal in the sense that every minimal $G$-flow is a factor of $\mb G$.

Proximal and strongly proximal flows will also play an important role. A $G$-flow $X$ is {\em proximal} if for any two points $x,y \in X$, there is a net $(g_i)$ in $G$ such that $\lim g_ix = \lim g_i y$. It is {\em strongly proximal} if for any probability measure $\mu \in \rP(X)$, there is a net $(g_i)$ in $G$ such that $\lim g_i \mu \in X$. Here we have identified points in $X$ with the corresponding Dirac measures in $\rP(X)$. There is a unique minimal proximal $G$-flow $\gb G$ that is universal in the sense that every minimal proximal $G$-flow is a factor of $\gb G$. Similarly, there is a unique universal minimal strongly proximal $G$-flow.

A $G$-flow $X$ is {\em point transitive} if there is a {\em transitive point} $x \in X$, i.e. a point such that the orbit $Gx = \{gx : g \in G\}$ is dense in $X$. The Stone--\v{C}ech compactification $\beta G$ is the unique universal point transitive $G$-flow, meaning that if $X$ is a $G$-flow with a transitive point $x$, then there is an extension $\alpha : \beta G \to X$ such that $\alpha(e) = x$. Here $e \in G$ denotes the unit in $G$, identified in a canonical way with an element of $\beta G$.

For a subset $A \subseteq G$, we will let $\ol{A}$ denote the closure of $A$ in $\beta G$. The set $\ol{A}$ is clopen, and every clopen subset of $\beta G$ is of this form. Moreover, the family of clopen subsets of $\beta G$ form a basis for the topology on $\beta G$. These facts follow from the fact that $\beta G$ can be realized as the Stone space of the Boolean algebra of subsets of $G$. We will discuss the relationship between $G$-flows and Boolean subalgebra of $G$ in more detail in Section \ref{sec:realizations}.

We will consider $\beta G$ as a compact right topological semigroup, meaning that for $y \in \beta G$ the map $\beta G \to \beta G : x \to xy$ is continuous. Many aspects of the structure of $\beta G$ are well understood (see e.g. \cite{HS2012}). For example, $\beta G$ contains idempotents and has minimal left ideals.

Let $L$ be a minimal left ideal in $\beta G$. Then $L$ is closed. Let $E \subseteq L$ denote the set of idempotents in $L$. Then $L$ decomposes as a disjoint union $L = \sqcup_{u \in E} uL$. In particular, $L$ contains idempotents, which are said to be minimal. Note that $L$ is a compact $G$-flow with respect to the left multiplication action of $G$. We will frequently use the fact that $L$ is isomorphic as a $G$-flow to the universal minimal proximal flow $\mb G$. 

For a $G$-flow $X$, a subset $S \subseteq X$ and a point $x \in X$, the corresponding {\em return set} $S_x \subseteq G$ is defined by $S_x = \{g \in G : gx \in S \}$. We will frequently make use of the following characterization of the closure of the return set of a clopen subset.

\begin{lem} \label{lem:return-set}
Let $X$ be a $G$-flow and let $U \subseteq X$ be a clopen subset. Then for $x \in X$, $\ol{U_x} = \{y \in \beta G : yx \in U\}$.
\end{lem}

\begin{proof}
Suppose that $y \in \ol{U_x}$. Then there is a net $(g_i)$ in $U_x$ such that $\lim g_i = y$. Then $\lim g_i x = yx$. Since $g_i x \in U$ for each $i$ and $U$ is closed, it follows that $yx \in U$.

Conversely, suppose that $yx \in U$. Let $(g_i)$ be a net in $G$ such that $\lim g_i = y$.  Then $\lim g_i x = yx \in U$. Since $U$ is open, eventually $g_i x \in U$. Hence eventually $g_i \in \ol{U_x}$ and therefore $y \in \ol{U_x}$.
\end{proof}

\subsection{Unital commutative C*-algebras}
The category of compact Hausdorff spaces with continuous maps as morphisms is dually equivalent to the category of unital commutative C*-algebras with unital *-homomorphisms as morphisms. For a compact Hausdorff space $X$, the corresponding dual object is the C*-algebra $\rC(X)$ of continuous functions on $X$.

If $Y$ is another compact Hausdorff space, then a continuous map $\alpha : Y \to X$ corresponds to the *-homomorphism $\pi : \rC(X) \to \rC(Y)$ defined by $\pi(f)(y) = f(\alpha(y))$ for $f \in \rC(X)$ and $y \in Y$. We will frequently use the fact that the map $\alpha$ is surjective if and only if the map $\pi$ is injective.

Let $G$ be a discrete group. We will say that the unital commmutative C*-algebra $\rC(X)$ is a $G$-C*-algebra if there is an action of $G$ on $\rC(X)$ by automorphisms. For $g \in G$ and $f \in \rC(X)$, we will write $gf$ for the image of $f$ under the automorphism corresponding to $f$.

The category of $G$-flows with equivariant continuous maps as morphisms is dually equivalent to the category of unital commutative $G$-C*-algebras. If $X$ is a $G$-flow, then the C*-algebra $\rC(X)$ is a $G$-C*-algebra with respect to the corresponding action of $G$ defined by $gf(x) = f(g^{-1}x)$ for $g \in G$ and $f \in \rC(X)$. We will discuss the relationship between $G$-flows and $G$-C*-algebras in more detail in Section \ref{sec:injectivity}. 

\section{Higher order syndeticity} \label{sec:higher-order-syndeticity}

\subsection{Higher order syndeticity}

Let $G$ be a discrete group. Recall that a subset $A \subseteq G$ is {\em (left) syndetic} if there is a finite subset $F \subseteq G$ such that $FA = G$. In this section we will introduce a ``higher-order'' notion of syndeticity for subsets of $G$ and consider several characterizations of sets with this property. In later sections we will establish connections with the topological dynamics of $G$. In particular, we will see that the structure of the completely syndetic subsets of a group is closely related to the problem of the existence of minimal proximal flows for the group.

\begin{defn} \label{defn:completely-syndetic}
Let $G$ be a discrete group. For $n \in \bN$, a subset $A \subseteq G$ is {\em (left) $n$-syndetic} if there is a finite subset $F \subseteq G$ such that $F A^n = G^n$.  We will say that $A$ is {\em completely (left) syndetic} if it is $n$-syndetic for all $n \in \bN$.
\end{defn}

The next result follows immediately from Definition \ref{defn:completely-syndetic}, and provides a useful reformulation of the definition of an $n$-syndetic subset.

\begin{lem} \label{lem:n-syndetic-immediate}
  Let $G$ be a discrete group and $n \in  \bN$.  A subset $A \subseteq G$ is $n$-syndetic if and only if there is a finite subset $F \subseteq G$ such that for every finite subset $K \subseteq G$ with $|K| = n$, there is $f \in F$ such that $fK \subseteq A$.
\end{lem}

The next proposition provides another characterization of $n$-syndetic subsets.

\begin{prop}
Let $G$ be a discrete group. A subset $A \subseteq G$ is $n$-syndetic if and only if there is a finite subset $F \subseteq G$ such that whenever $F$ is partitioned as $F = F_1 \sqcup \cdots \sqcup F_n$, then $F_i A = G$ for some $i$.
\end{prop}

\begin{proof}
For $n \in \bN$ and a finite subset $F \subseteq G$, 
\begin{align*}
&(F(A^n))^c = \bigcap_{f \in F} f (A^n)^c \\
&\quad = \bigcap_{f \in F} f((A^c \times G^{n-1}) \cup (G \times A^c \times G^{n-2}) \cup \cdots \cup (G^{n-1} \times A^c)) \\
&\quad = \bigcup_{F = F_1 \sqcup \cdots \sqcup F_n} \left( \bigcap_{f \in F_1} f A^c \times \cdots \times \bigcap_{f \in F_n} f A^c \right) \\
&\quad = \bigcup_{F = F_1 \sqcup \cdots \sqcup F_n} (F_1 A)^c \times \cdots \times (F_n A)^c,
\end{align*}
where the union is taken over all partitions $F = F_1 \sqcup \cdots \sqcup F_n$, with the convention that if $F_i = \emptyset$ then $\cap_{f \in F_i} fA^c = G$ and $F_i A = \emptyset$. It follows that $FA^n = G^n$ if and only if whenever $F$ is partitioned as $F = F_1 \sqcup \cdots \sqcup F_n$, then $F_i A = G$ for some $i$.
\end{proof}

Recall that a subset $B \subseteq G$ is {\em (right) thick} if for every finite subset $F \subseteq G$ there is $h \in G$ such that $Fh \subseteq B$.  One can see right from the definitions that a subset $A \subseteq G$ is syndetic if and only if $A^c$ is not thick.  This correspondence can be expressed in terms of dual families.  

Let $\cF$ be a family of subsets of $G$. The {\em dual family} $\cF^*$ is the family of subsets of $G$ having nonzero intersection with every element of $\cF$, i.e.
\[
\cF^* = \{B \subseteq G : A \cap B \ne \emptyset \text{ for all } A \in \cF \}.
\]
An important property of the family of thick subsets of $G$ is that it is dual to the family of syndetic subsets of $G$.  This fact will be generalised to $n$-syndetic subsets in Corollary \ref{cor:dual-n-syndetic} below.

We first identify the complements of $n$-syndetic subsets.

\begin{defn}
Let $G$ be a discrete group. For $n \in \bN$, we will say that a subset $B \subseteq G$ is {\em (right) $1/n$-thick} if for every finite subset $F \subseteq G$ there is $(h_1,\ldots,h_n)\in G^n$ such that
\[
F(h_1,\ldots,h_n) \subseteq \bigcup_{i=1}^n G^{n-i} \times B \times G^{i-1}.
\]
We will say that $B$ is {\em fractionally (right) thick} if it is $1/n$-thick for some $n \in \bN$. 
\end{defn}

\begin{prop} \label{prop:dual-n-syndetic}
Let $G$ be a discrete group. For $n \in \bN$, a subset $A \subseteq G$ is $n$-syndetic if and only if $A^c$ is not $1/n$-thick. Hence $A$ is completely syndetic if and only if $A^c$ is not fractionally thick.
\end{prop}

\begin{proof}
For a finite subset $F \subseteq G$, 
\[
(FA^n)^c = \cap_{f \in F} f(A^n)^c = \bigcap_{f \in F} f \left( \bigcup_{i=1}^n G^{n-i} \times A^c \times G^{i-1} \right).
\]
Therefore, $FA^n \ne G^n$ if and only if the intersection on the right is non-empty. This is equivalent to the condition that there is $(h_1,\ldots,h_n) \in G^n$ such that
\[
F^{-1}(h_1,\ldots,h_n) \subseteq \bigcup_{i=1}^n G^{n-i} \times A^c \times G^{i-1}.
\]
\end{proof}

\begin{cor} \label{cor:dual-n-syndetic}
Let $G$ be a discrete group. For $n \in \bN$, the dual to the family of $n$-syndetic subsets is the family of $1/n$-thick sets. Hence the dual to the family of completely syndetic subsets is the family of fractionally thick subsets.
\end{cor}

\begin{proof}
Suppose $B \subseteq G$ is a subset satisfying $A \cap B \ne \emptyset$ for all $n$-syndetic subsets $A \subseteq G$. Then since $B \cap B^c = \emptyset$, $B^c$ is not $n$-syndetic. Hence by Proposition \ref{prop:dual-n-syndetic}, $B$ is $1/n$-thick. 

Conversely, suppose that $A \cap B = \emptyset$ for some $n$-syndetic subset $A$. Then $A \subseteq B^c$, and since any set containing an $n$-syndetic subset is itself $n$-syndetic, $B^c$ is $n$-syndetic. Hence by Proposition \ref{prop:dual-n-syndetic}, $B$ is not $1/n$-thick.
\end{proof}

\begin{example} \label{ex:integers}
It is well known that a subset $A \subseteq \bZ$ is syndetic if and only if it has ``bounded gaps'' meaning that there is $k \in \bN$ such that for any $a \in \bZ$, $\{a, a + 1, \ldots, a + k\} \cap A \ne \emptyset$. 

For $n \in \bN$, a subset $A \subseteq \bZ$ is $n$-syndetic if and only if $A^n$ has ``bounded diagonal gaps" meaning that there is $k \in \bN$ such that for any $n$-tuple $(a_1,\dots,a_n) \in \bZ^n$, 
\begin{equation}
\{(a_1,\dots,a_n), (a_1+1,\dots,a_n+1) \dots, (a_1+k,\dots,a_n+k)\} \cap A^n \ne \emptyset. \label{eq:ex-integers}
\end{equation}
To see this recall that $A$ is $n$-syndetic if and only if $A^c$ is not $1/n$-thick, meaning that there is a finite subset $F \subseteq \mathbb Z$ such that for every $n$-tuple $(a_1,\dots,a_n) \in \bZ^n$, $(F+(a_1,\dots,a_n)) \not \subseteq  \bigcup_{i=1}^n \bZ^{n-i} \times A^c \times \bZ^{i-1}$. Equivalently, $(F+(a_1,\dots,a_n)) \cap A^n \neq \emptyset$. Therefore, if $A$ is $n$-syndetic and $F \subseteq \bZ$ satisfies $F+A^n = \bZ^n$, then (\ref{eq:ex-integers}) is satisfied for $k=\max(F)-\min(F)+1$. Conversely, if there is $k \in \bN$ such that (\ref{eq:ex-integers}) is satisfied, then for $F = \{0,1,\ldots,k\}$, $F+A^n = \bZ^n$, showing that $A$ is $n$-syndetic.

Since a subgroup is syndetic if and only if it is finite index, and since every subgroup of $\bZ$ is of the form $m\bZ$ for $m \in \bN$, and so is of finite index, every subgroup of $\bZ$ is syndetic. However, no subgroup of $\bZ$ is $2$-syndetic since, for any $m \in \bN \setminus \{1\}$, there are arbitrarily long diagonal segments in $\bZ^2$ that do not intersect $(m\bZ)^2$. This can be seen, for example, in Figure~\ref{fig:1}.

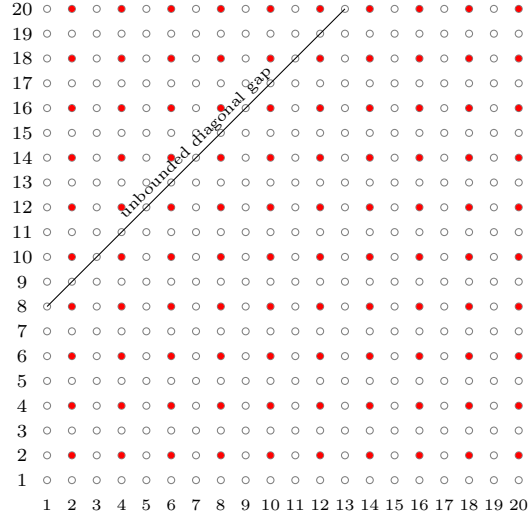
\begin{figure}[H]
\centering
  \begin{tikzpicture}[align=center,scale=0.33]
    \foreach \i in {1,...,20}
      \path[black] (\i,0) node{\fontsize{6}{8}\selectfont{\i}} (0,\i) node{\fontsize{7}{8}\selectfont{\i}};
    \foreach \i in {1,...,20}
      \foreach \j in {1,...,20}{
        \draw[gray] (\i,\j) circle(4pt);
       \ifodd \i \else
          \ifodd \j \else
            \fill[red] (\i,\j) circle(3.7pt);
          \fi
        \fi
      };
    \draw (1,8) -- node[above=-28pt]{\rotatebox{45}{\fontsize{6}{8}\selectfont{unbounded diagonal gap}}} (13,20) ;
  \end{tikzpicture}
  \caption{The subgroup $2\bZ \subseteq \bZ$ is syndetic but not $2$-syndetic since there are arbitrarily long diagonal segments in $\bZ \times \bZ$ that do not intersect $2\bZ \times 2\bZ$.} \label{fig:1}
\end{figure}

Now fix $n \in \bN$ with $n > 2$. We will show that the set $A=\bZ \setminus n\bZ$ is $(n-1)$-syndetic but not $n$-syndetic. To see that $A$ is $(n-1)$-syndetic, first note that for $a \in \bZ$, there is exactly one multiple of $n$ in the set $\{a,a+1 \dots, a+(n-1)\}$. Hence for an $(n-1)$-tuple $(a_1,\dots, a_{n-1}) \in \bZ^{n-1}$, at most $n-1$ elements of the set 
\[
\{(a_1,\dots, a_{n-1}),(a_1+1,\dots, a_{n-1}+1), \dots,(a_1+(n-1),\dots, a_{n-1}+(n-1))\}
\]
have an entry that is a multiple of $n$.  Therefore, at least one element in the set has no entries that are multiples of $n$, and this element belongs to $A^{n-1}$. 

To see that $A$ is not $n$-syndetic, note that for $k\in \bN$, every member of the set
\[
\{(1,2,\dots, n),(2,3,\dots, n+1), \dots, (1+k,2+k,\dots, n+k)\}
\]
has an entry that is a multiple of $n$. Thereore, the set does not intersect $A^n$.

\begin{figure}[H]
\centering
  \begin{tikzpicture}[align=center,scale=0.33]
    \foreach \i in {1,...,20}
      \path[black] (\i,0) node{\fontsize{6}{8}\selectfont{\i}} (0,\i) node{\fontsize{7}{8}\selectfont{\i}};
    \foreach \i in {1,...,20}
      \foreach \j in {1,...,20}{
        \draw[gray] (\i,\j) circle(4pt);
             \fill[red] (\i,\j) circle(3.7pt);
      };
      \foreach \j in {1,...,20}{
     \fill[white] (3,\j) circle(3.8pt);
     \fill[white] (6,\j) circle(3.8pt);
     \fill[white] (9,\j) circle(3.8pt);
     \fill[white] (12,\j) circle(3.8pt);
     \fill[white] (15,\j) circle(3.8pt);           
     \fill[white] (18,\j) circle(3.8pt);           
      };      
      \foreach \i in {1,...,20}{
     \fill[white] (\i,3) circle(3.8pt);
     \fill[white] (\i,6) circle(3.8pt);
     \fill[white] (\i,9) circle(3.8pt);
     \fill[white] (\i,12) circle(3.8pt);
     \fill[white] (\i,15) circle(3.8pt);           
     \fill[white] (\i,18) circle(3.8pt);           
      };      
  \end{tikzpicture}
 \caption{The subset $\bZ \setminus 3\bZ \subseteq \bZ$ is $2$-syndetic but not $3$-syndetic since for $k\in \mathbb N$, every element in the set $\{(1,2,3),(2,3,4),(4,5,6), \dots, (1+k,2+k,3+k)\}$ has an entry that is a multiple of 3, implying that the set does not intersect $A^3$.} \label{fig:2}
\end{figure}
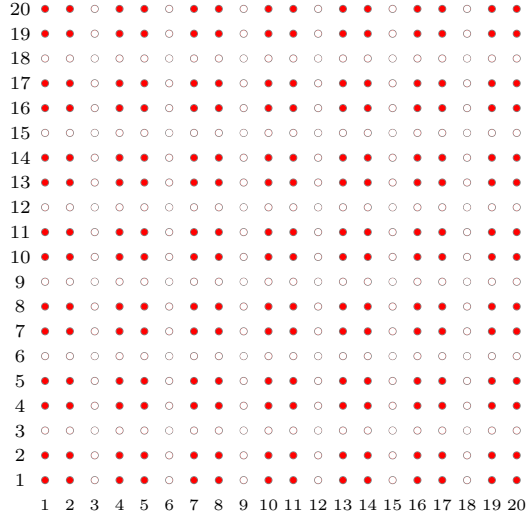

Finally, we will construct an example of a non-trivial completely syndetic set. Let $B$ denote the complement of the set of powers of $2$ in $\bZ$. Fix $n\in \bN$, and let $k=2^{n-1}+1$. Then for $a \in \bZ$, there are at most $n$ powers of $2$ in the subset $\{a,a+1, \dots, a+k\}$. Hence, for every $n$-tuple $(a_1,\dots, a_n) \in \bZ^n$ there are at most $n^2$ elements of the set
\[
\{(a_1,\dots, a_n),(a_1+1,\dots, a_n+1), \dots,(a_1+k,\dots, a_n+k)\}
\]
with an entry that is a power of $2$. Since $n^2<2^n+2=k+1$, this implies that at least one element of the set has no entries that are a power of $2$, and this element belongs to $B^n$. Therefore, $B$ is $n$-syndetic. Since $n$ was arbitary, it follows that $B$ is completely syndetic.

\begin{figure}[H]
\centering
  \begin{tikzpicture}[align=center,scale=0.33]
    \foreach \i in {1,...,20}
      \path[black] (\i,0) node{\fontsize{6}{8}\selectfont{\i}} (0,\i) node{\fontsize{7}{8}\selectfont{\i}};
    \foreach \i in {1,...,20}
      \foreach \j in {1,...,20}{
        \draw[gray] (\i,\j) circle(4pt);
             \fill[red] (\i,\j) circle(3.7pt);
      };
      \foreach \j in {1,...,20}{
     \fill[white] (1,\j) circle(3.8pt);
     \fill[white] (2,\j) circle(3.8pt);
     \fill[white] (4,\j) circle(3.8pt);
     \fill[white] (8,\j) circle(3.8pt);
     \fill[white] (16,\j) circle(3.8pt);           
      };      
      \foreach \i in {1,...,20}{
     \fill[white] (\i,1) circle(3.8pt);
     \fill[white] (\i,2) circle(3.8pt);
     \fill[white] (\i,4) circle(3.8pt);
     \fill[white] (\i,8) circle(3.8pt);
     \fill[white] (\i,16) circle(3.8pt);           
      };      
  \end{tikzpicture}
  \caption{The complement of the set of powers of $2$ in $\bZ$ is completely syndetic, and in particular is $2$-syndetic.} \label{fig:3}
\end{figure}
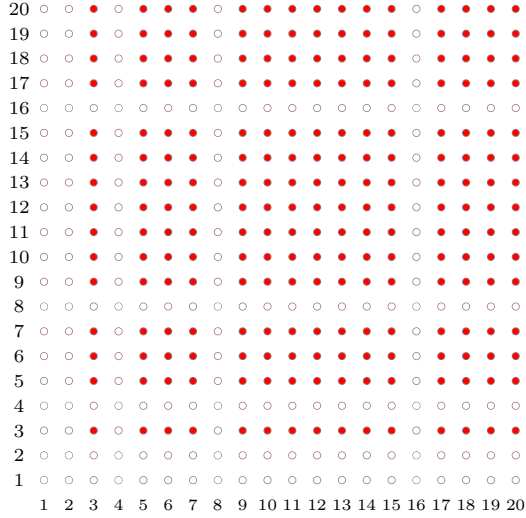

\end{example}

The next result is a generalization of the criterion that was used in Example \ref{ex:integers}.

\begin{lem} \label{lem:n-syndetic-intersection}
Let $G$ be a discrete group. For $n \in \bN$, a subset $A \subseteq G$ is $n$-syndetic if and only if there is a finite subset $F \subseteq G$ such that for any finite subset $K \subseteq G$ with $|K| = n$, $F \cap (\cap_{k \in K} Ak) \ne \emptyset$.
\end{lem}

\begin{proof}
If $A$ is $n$-syndetic then there is a finite subset $F \subseteq G$ such that $FA^n = G^n$. For finite $K \subseteq G$ with $|K| = n$, write $K=\{k_1,\ldots, k_n\}$. Then there is $f \in F$ such that $(k_1^{-1},\ldots,k_n^{-1}) \in fA^n$. Equivalently, $f^{-1} \in \cap_{i=1}^n Ak_i$. In particular, $F^{-1} \cap (\cap_{k \in K} Ak) \ne \emptyset$.

Conversely, if there is a finite subset $F \subseteq G$ such that for any finite subset $K \subseteq G$ with $|K| = n$, $F \cap (\cap_{k \in K} Ak) \ne \emptyset$, then it is easy to check that $F^{-1}A^n = G^n$. Hence $A$ is $n$-syndetic.
\end{proof}

\begin{example}
  Let $G$ be a discrete group. Let $A,B \subseteq G$ be subsets with $A$ syndetic and $B$ thick. We claim that the product set $AB$ is completely syndetic. To see this, choose a finite subset $F \subseteq G$ such that $FA = G$.  Given $K \subseteq G$ finite, the assumption on $B$ implies that there is $h \in \cap_{k \in K} Bk \ne \emptyset$.  Then Lemma \ref{lem:n-syndetic-intersection} applied for $n = 1$, says that $F^{-1} \cap Ah \ne \emptyset$.  Therefore,
\[
\emptyset \ne F^{-1} \cap Ah \subseteq  F^{-1} \cap A(\cap_{k \in K} Bk) \subseteq F^{-1} \cap (\cap_{k \in K} ABk),
\]
so Lemma \ref{lem:n-syndetic-intersection} implies that $AB$ is completely syndetic.
\end{example}

In later sections we will be interested in subsets $A \subseteq G$ that ``avoid'' a given finite subset of $F \subseteq G$, in the sense that $FA \cap A = \emptyset$. For $n \in \bN$, let $2^G_n$ denote the family of subsets of $G$ of size $n$. Then $n$-syndetic subsets of $G$ with this property correspond to certain ``colorings'' of $2^G_n$.

\begin{defn}
Let $G$ be a countable discrete group. For a finite subset $F \subseteq G$ and $n \in \bN$, an $(F,n)$-coloring of $G$ is a pair $(K,k)$ consisting of a finite subset $K \subseteq G$ and a function $k : 2^G_n \to K$ such that for any pair of subsets $E_1,E_2 \in 2^G_n$, we have $F k(E_1)E_1 \cap k(E_2)E_2 = \emptyset$.
\end{defn}

\begin{prop} \label{prop:coloring}
 Let $G$ be a discrete group and $F \subseteq G$ finite. There is a surjection from $(F,n)$-colorings of $G$ onto $n$-syndetic subsets $A \subseteq G$ satisfying $FA \cap A = \emptyset$. For an $(F,n)$-coloring $(K,k)$, the corresponding $n$-syndetic subset is $A = \bigcup_{E \in 2^G_n} k(E)E$.
\end{prop}
\begin{proof}
Let $(K,k)$ be an $(F,n)$-coloring of $G$ and let $A = \cup_{E \in 2^G_n} k(E) E$. Then by the definition of an $(F,n)$-coloring,
\[
FA \cap A \subseteq \bigcup_{E_1,E_2 \in 2^G_n} F k(E_1)E_1 \cap k(E_2)E_2 = \emptyset.
\]
Furthermore, for $E \in 2^G_n$, we have the inclusion $k(E)E \subseteq A$, implying that $A$ is $n$-syndetic.

Conversely, let $A \subseteq G$ be an $n$-syndetic subset such that $FA \cap A = \emptyset$. Then there is a finite subset $K \subseteq G$ such that $KA^n = G^n$. Write $K = \{k_1,\ldots,k_m\}$ and define $k : 2_n^G \to K^{-1}$ by $k(E) = k_i^{-1}$ for $E \subseteq 2_n^G$, where $i = \min\{i : E \subseteq k_i A \}$. Then for $E_1,E_2 \subseteq 2^{G^n}$, the fact that $k(E_1) E_1,k(E_2) E_2 \subseteq A$ implies that $F k(E_1) E_1 \cap k(E_2) E_2 = \emptyset$. Therefore, the pair $(k,K^{-1})$ is an $(F,n)$-coloring of $G$.
\end{proof}

\begin{rem}
Since the choice of the enumeration of $K$ in the second part of the proof of Proposition \ref{prop:coloring} is arbitrary, this correspondence will not generally be a bijection.
\end{rem}

\subsection{Stone--\v{C}ech compactification}

Let $G$ be a discrete group. There is an important characterization of syndetic subsets in terms of the Stone--\v{C}ech compactification $\beta G$ of $G$. Namely, a subset $A \subseteq G$ is syndetic if and only if for every minimal left ideal $L$ of $\beta G$, $\ol{A} \cap L \ne \emptyset$. In this section we will establish a characterization of higher order syndetic subsets in terms of $\beta G$. This characterization will be important in later sections when we consider the topological dynamical structure of $G$. 

For a subset $A \subseteq G$ and a point $x \in \beta G$, the corresponding {\em return set} is $\ol{A}_x = \{g \in G : gx \in \ol{A} \}$.

\begin{prop} \label{prop:n-syndetic-thickness}
Let $G$ be a discrete group. A subset $A \subseteq G$ is $n$-syndetic if and only if for every finite subset $K \subseteq \beta G$ with $|K| \leq n$, there is $g \in G$ such that $gK \subseteq \ol{A}$. Equivalently, $\cap_{x \in K} \ol{A}_x \ne \emptyset$.
\end{prop}

\begin{proof}
Suppose that $A$ is $n$-syndetic. Then there is a finite subset $F \subseteq G$ such that $FA^n = G^n$, and taking closures on both sides with respect to the product topology implies $F\ol{A}^n = (\beta G)^n$. Let $K \subseteq \beta G$ be a finite subset with $|K| \leq n$. Then from above, $K \subseteq f \ol{A}$ for some $f \in F$. Hence $f^{-1} K \subseteq \ol{A}$.

Conversely, suppose that for every finite subset $K \subseteq \beta G$ with $|K| \leq n$, there is $g \in G$ such that $gK \subseteq \ol{A}$. Then $\bigcup_{g \in G} g \ol{A}^n = (\beta G)^n$. Since $\ol{A}$ is clopen and $\beta G$ is compact, $\ol{A}^n$ is clopen and $(\beta G)^n$ is compact with respect to the product topology. Hence there is a finite subset $F \subseteq G$ such that $F \ol{A}^n = (\beta G)^n$. It follows that $FA^n = G^n$, so $A$ is $n$-syndetic.
\end{proof}

\begin{prop} \label{prop:left-thick-stone-cech}
Let $G$ be a discrete group. A subset $A \subseteq G$ is completely syndetic if and only if the closure $\ol{A} \subseteq \beta G$ contains a closed right ideal of $\beta G$.
\end{prop}

\begin{proof}
Suppose that $\ol{A}$ contains a right ideal $R$. Then for $x \in R$ and a finite subset $K \subseteq \beta G$, $xK \subseteq R \subseteq \ol{A}$. Let $(g_i)$ be a net in $G$ such that $\lim g_i = x$. Then since $\ol{A}$ is clopen, eventually $g_i K \subseteq \ol{A}$, since the product on $\beta G$ is continuous in the first variable. Hence by Proposition \ref{prop:n-syndetic-thickness}, $A$ is completely syndetic. 

Conversely, suppose that $A$ is completely syndetic. Then by Proposition \ref{prop:n-syndetic-thickness}, the family of sets $\{\ol{A}_y\}_{y \in \beta G}$ has the finite intersection property. Hence the family $\{\ol{\ol{A}_y}\}_{y \in \beta G}$ of clopen subsets also has the finite intersection property. Therefore, by the compactness of $\beta G$, there is $x \in \cap_{y \in \beta G} \ol{\ol{A}_y}$. Then $xy \in \ol{A}$ for all $y \in \beta G$, implying that the right ideal $R = x\beta G$ satisfies $R \subseteq \ol{A}$.

If $\ol{A}$ contains a right ideal, then since the closure of every right ideal in $\beta G$ is a right ideal (see e.g. \cite{HS2012}*{Theorem 2.15}), $\ol{A}$ necessarily contains a closed right ideal.
\end{proof}

If $G$ is amenable, then syndetic subsets of $G$ can be characterized in terms of left invariant means on $G$. Specifically, a subset $A \subseteq G$ is syndetic if and only if for every left invariant mean $m$ on $G$, $m(A) > 0$ (see e.g. \cite{BHM1998}*{Theorem 2.7}). The next result generalizes this characterization.

\begin{lem}
Let $G$ be an amenable discrete group. For $n \in \bN$,  a subset $A \subseteq G$ is $n$-syndetic if and only if for every mean $m$ on $G^n$ that is invariant under left translation by $G$ we have $m(A^n) > 0$.
\end{lem}

\begin{proof}
Suppose that $A$ is $n$-syndetic and let $m$ be a mean on $G^n$ that is invariant under left translation by $G$. Since $A$ is $n$-syndetic there is a finite subset $F \subseteq G$ such that $FA^n = G^n$. Then $m(FA^n) = 1$. By the translation invariance of $m$, it follows that $m(A^n) > 0$. 

Conversely, suppose that $A \subseteq G$ is not $n$-syndetic. Then by Proposition \ref{prop:dual-n-syndetic}, for every finite subset $F \subseteq G$ there is $h_F = (h_1,\ldots,h_n) \in G^n$ such that $Fh_F \subseteq (A^n)^c$. The net $(h_F)_{F \subseteq G}$ indexed by finite subsets $F \subseteq G$ is directed by inclusion. Let $x \in \beta(G^n)$ be a cluster point. Then $gx \in \ol{(A^n)^c}$ for all $g \in G$.

From above, letting $\delta_x$ denote the Dirac measure on $\beta(G^n)$ corresponding to $x$, $(g\delta_x)(\ol{A^n}) = 0$ for all $g \in G$. Since $G$ is amenable there is a probability measure $\mu$ on $\beta(G^n)$ contained in the closed convex hull of the set $\{g \delta_x\}_{g \in G}$ that is invariant under left translation by $G$. By construction, $\mu(\ol{A^n}) = 0$. Therefore, letting $m$ denote the corresponding mean on $G^n$, $m$ is invariant under left translation by $G$ and satisfies $m(A^n) = 0$.
\end{proof}

\subsection{Subshift dynamics} \label{sec:subshift-dynamics}

Let $G$ be a countable discrete group and let $2^G$ denote the family of subsets of $G$. Fix an enumeration of $G$ as $G = \{g_k : k \in \bN \}$ and equip $2^G$ with the metric $d : 2^G \times 2^G \to \bR$ defined by $d(A,B) = 1/n$ for $A,B \in 2^G$, where $n = \inf \{k \in \bN : A \cap \{g_k\} \ne B \cap \{g_k\} \}$. Then $2^G$ is compact and the right translation action on $2^G$ defines a $G$-flow called the {\em Bernoulli shift}.

A subflow of the Bernoulli shift is called a {\em subshift}. In this section we will establish a characterization of higher order syndetic subsets in terms of a property of the subshifts that they generate. We will apply this result in later sections when we discuss strong amenability.

For a subset $A \subseteq G$, let $X \subseteq 2^G$ denote the subshift generated by $A$. Then it is well known that $A$ is syndetic if and only if $\emptyset \notin X$ and $A^c$ is syndetic if and only if $G \notin X$. The next result is a generalization of this fact.

\begin{prop} \label{prop:subshift-dynamics-n-syndetic}
Let $G$ be a countable discrete group, let $A \subseteq G$ be a subset of $G$ and let $X \subseteq 2^G$ denote the subshift generated by $A$. Then for $n \in \bN$, $A$ is $n$-syndetic if and only if for any $B_1,\ldots,B_n \in X$, $B_1 \cap \cdots \cap B_n \ne \emptyset$. Also, $A^c$ is $n$-syndetic if and only if for any $B_1,\ldots,B_n \in X$, $B_1 \cup \cdots \cup B_n \ne G$.
\end{prop}

\begin{proof}
Suppose that $A$ is $n$-syndetic. Then by Lemma \ref{lem:n-syndetic-intersection} there is a finite subset $F \subseteq G$ such that for any finite subset $K \subseteq G$ with $|K| = n$, $F \cap (\cap_{k \in K} Ak) \ne \emptyset$. Choose $B_1,\ldots,B_n \in X$ and for each $i \in \{1, \dotsc, n\}$ let $(g_k^i)$ be a sequence in $G$ such that $\lim_k A g_k^i = B_i$. Then from above, $F \cap (\cap_{i=1}^n Ag_k^i) \ne \emptyset$ for each $k$. Since $F$ is finite, by passing to subsequences we can assume that there is $f \in F$ such that $f \in A g_k^i$ for all $k$ and $i$. Then $f \in B_1 \cap \cdots \cap B_n$, and in particular $B_1 \cap \cdots \cap B_n \ne \emptyset$.

Conversely, suppose there is $B_1,\ldots,B_n \in X$ such that $B_1 \cap \cdots \cap B_n = \emptyset$. Let $(g_k^i)$ be sequences in $G$ such that $\lim_k A g_k^i = B_i$. Then for any finite subset $F \subseteq G$, eventually $F \cap (\cap_{i=1}^n A g_k^i) = \emptyset$. Indeed, otherwise we could argue as above that $B_1 \cap \cdots \cap B_n \ne \emptyset$.
Hence by Lemma \ref{lem:n-syndetic-intersection}, $A$ is not $n$-syndetic.

For subsets $B_1,\ldots,B_n \subseteq G$, $B_1,\ldots,B_n$ are in the subshift generated by $A$ if and only if $B_1^c,\ldots,B_n^c$ are in the subshift generated by $A^c$. It follows from above that $A^c$ is $n$-syndetic if and only if for any $B_1,\ldots,B_n \in X$, $B_1 \cup \cdots \cup B_n \ne G$. 
\end{proof}

\subsection{Strongly complete syndeticity}

In this section we will introduce a slightly stronger notion of complete syndeticity for subsets of a discrete group inspired by the proof of \cite{Gla1976}*{Chapter VII, Proposition~2.1}. We will see later that, just as the structure of the completely syndetic subsets of a group is closely related to the problem of the existence of minimal proximal flows for the group, the structure of the subsets satisfying this stronger property is closely related to the problem of the existence of minimal strongly proximal flows.

Recall that a multiset is a set with multiplicity. Formally, a multiset is a pair $X = (A,m)$ consisting of a set $A$ and a multiplicity function $m : A \to \bN_{\geq 1}$. The cardinality of $X$ is $|X| = \sum_{a \in A} m(a)$. For an ordinary set $Y$, we will write $X \subseteq Y$ if $A \subseteq Y$, and we will write $X \cap Y$ for the multiset $(A \cap Y, m|_{A \cap Y})$. Finally, for a discrete group $G$, a multiset $X = (A,m) \subseteq G$ and $g \in G$, we will write $gX = (gA, gm)$, where $gm : gA \to \bN_{\geq 1}$ is defined by $gm(ga) = m(a)$ for $a \in A$. Note that if $X$ is an ordinary set, identified with a multiset in the obvious way, then these definitions will coincide with the usual definitions.

\begin{defn} \label{def:strongly-completely-syndetic}
Let $G$ be a discrete group. We will say that a subset $A \subseteq G$ is {\em strongly completely (left) syndetic} if for every $\epsilon > 0$ there is a finite subset $F \subseteq G$ such that for every finite multiset $K \subseteq G$, there is $f \in F$ such that $|fK \cap A| \geq (1-\epsilon)|K|$. 
\end{defn}

\begin{rem}
We will see in Section \ref{sec:universal-minimal-strongly-proximal-flow} that it is essential that the set $K$ in the above definition is allowed to be a multiset.
\end{rem}

\begin{lem}
Let $G$ be a discrete group. If a subset $A \subseteq G$ is strongly completely syndetic then it is completely syndetic.
\end{lem}

\begin{proof}
Suppose that $A$ is strongly completely syndetic. For $n \in \bN$, let $\epsilon = \frac{1}{2n}$. Then there is a finite subset $F \subseteq G$ such that for every finite subset $K \subseteq G$ with $|K| = n$ there is $f \in F$ such that $|fK \cap A| \geq (1 - \epsilon)|K| = n - \frac{1}{2}$. Hence $|fK \cap A| = n$, implying $fK \subseteq A$. Therefore, $A$ is left $n$-syndetic.
\end{proof}
We will see later that not every completely syndetic subset is strongly completely syndetic.

\begin{example} \label{ex:scs-free-group}
Consider the free group on two generators $\bF_2 = \langle a,b \rangle$ and let $A \subseteq \bF_2$ denote the set of all elements in $\bF_2$ with reduced form beginning with $a$. We will show that $A$ is strongly completely syndetic.  To this end choose $\epsilon > 0$ and $n \in \bN$ with $n \geq 1/(2\epsilon)$. 

For a reduced word $w$ in $\{a,b\}$, let $B_w \subseteq \bF_2$ denote the subset of elements such that the corresponding reduced word begins with $w$. Then there are reduced words $w_2,\ldots,w_{2n}$ such that setting $w_1 = a$, $\bF_2 = B_a \sqcup B_{w_1} \sqcup \cdots \sqcup B_{w_{2n}}$ and there is a finite subset $F \subseteq \bF_2$ with the property that for each $i$, there is $f \in F$ satisfying $fB_{w_j} \subseteq B_a$ if $j \ne i$.

For example, if $n = 2$, we can take $w_2 = a^{-1}$, $w_3 = b$, $w_4 = b^{-1}$ and $F = \{a,ab\}$. For $n = 3$, we can take $w_2 = a^{-2}$, $w_3 = a^{-1}b$, $w_4 = a^{-1}b^{-1}$, $w_5 = b$, $B_6 = b^{-1}$ and $F = \{a^2,ab\}$. 

Then for a finite multiset $K \subseteq G$, $|K \cap B_{w_i}| \leq |K|/(2n)$ for some $i$. By construction there is $f \in F$ such that $fB_{w_j} \subseteq B_a$ for $j \ne i$. Then
\[
|fK \cap A| \geq (1-1/(2n))|K| \geq (1-\epsilon)|K|.
\]
Therefore, $A$ is strongly completely syndetic.
\end{example}

\section{Algebraically irreducible affine flows}

Let $G$ be a discrete group and let $X$ be a $G$-flow. If $X$ is minimal and strongly proximal, then the affine flow $\rP(X)$ of probability measures on $X$ is {\em irreducible}, meaning that it has no proper affine subflow. Indeed, an affine subflow of $\rP(X)$ necessarily contains $X$, and hence by the Krein-Milman theorem, contains all of $\rP(X)$. Furstenberg proved (see e.g. \cite{Gla1976}*{Theorem 2.3}) a kind of converse: if $K$ is any irreducible affine $G$-flow, then the closure of the set of extreme points of $K$ is a minimal strongly proximal $G$-flow.

Zorn's lemma implies that every affine $G$-flow $L$ contains an irreducible affine $G$-flow $K$. It follows from above that for a $G$-flow $X$, the affine flow $\rP(X)$ always contains a minimal strongly proximal $G$-flow. Consequently, there is always a continuous equivariant map from the universal minimal strongly proximal $G$-flow $\fb G$ to $\rP(X)$. The existence of this map along with more general ``boundary maps'' has played an important role in recent applications of topological dynamics to C*-algebras (see e.g. \cite{KK2017} and \cite{Ken2018}). 

In this section we will establish a similar result for minimal proximal flows. We will utilize this result in later sections. However, in general it does not seem to be as useful as the corresponding result for minimal strongly proximal flows.

\begin{defn}
Let $G$ be a discrete group and let $K$ be an affine $G$-flow. We will say that $K$ is {\em algebraically irreducible} if whenever $L \subseteq K$ is an affine $G$-subflow with the property that $L \cap \operatorname{conv}(\ol{\partial K}) \ne \emptyset$, then $L = K$.
\end{defn}

\begin{prop} \label{prop:alg-irred-unique-min-subflow}
Let $G$ be a discrete group, let $K$ be an algebraically irreducible affine $G$-flow and let $X = \ol{\partial K}$. Then $X$ is a minimal proximal $G$-flow and it is the unique minimal subflow of $K$ with the property that $X \cap \operatorname{conv}(\ol{\partial K}) \ne \emptyset$.
\end{prop}

\begin{proof}
To see that $X$ is minimal and proximal, choose $x, y \in X$ and let $w = \frac{1}{2}(x + y) \in K$. Let $L \subseteq K$ denote the closed convex hull of the orbit $Gw$. Then $L$ is an affine subflow with $L \cap \operatorname{conv}(\ol{\partial K}) \ne \emptyset$. Since $K$ is algebraically irreducible, $L = K$. Hence by Milman's partial converse to the Krein-Milman theorem, $\partial K \subseteq \ol{Gw}$. It follows that $X \subseteq \ol{Gw}$. Taking $x = y$ implies that $X$ is minimal.

To see that $X$ is proximal, choose a point $z \in \partial K$. Then from above there is a net $(g_i)$ in $G$ such that $z = \lim g_i w = \lim \frac{1}{2}(g_i x + g_i y)$. Since $X$ is compact, by passing to a subnet we can assume there are points $x',y' \in X$ such that $\lim g_i x = x'$ and $\lim g_i y = y'$. Then $z = \frac{1}{2}(x' + y')$. Since $z$ is an extreme point, it follows that $x' = y' = z$. Hence $X$ proximal. 

Finally, to see that $X$ is unique, let $Y \subseteq K$ be a minimal subflow of $K$ such that $Y \cap \operatorname{conv}(\ol{\partial K}) \ne \emptyset$. Then the closed convex hull $L$ of $Y$ is an affine $G$-flow with $L \cap \operatorname{conv}(\ol{\partial K}) \ne \emptyset$. Since $K$ is algebraically irreducible, $L = K$. Hence by Milman's partial converse to the Krein-Milman theorem, $\partial K \subseteq Y$. Since $\ol{\partial K}$ is invariant, the minimality of $Y$ implies that $Y = \ol{\partial K} = X$.
\end{proof}

\begin{thm} \label{thm:min-prox-iff-alg-irred}
Let $G$ be a discrete group. A $G$-flow $X$ is minimal and proximal if and only if the affine $G$-flow $\rP(X)$ is algebraically irreducible.
\end{thm}

\begin{proof}
Suppose that $\rP(X)$ is algebraically irreducible. Then Proposition \ref{prop:alg-irred-unique-min-subflow} implies $X = \partial \rP(X)$ is minimal and proximal.  Conversely, suppose that $X$ is minimal and proximal. Let $K \subseteq \rP(X)$ be an affine $G$-flow with $K \cap \co(X) \ne \emptyset$. Then there is $\mu \in K$ of the form $\mu = \alpha_1 \delta_{x_1} + \cdots + \alpha_n \delta_{x_n}$ for $\alpha_1,\ldots,\alpha_n > 0$ and $\alpha_1 + \cdots + \alpha_n = 1$ and $x_1,\ldots,x_n \in X$. By the proximality of $x$ there is a net $(s_i)$ in $G$ and $y \in X$ such that $\lim s_i x_i = y$ for all $i$. Hence $\lim s_i \mu = \delta_y$. Thus $X \subseteq K$ and  $K = \rP(X)$. Hence $\rP(X)$ is algebraically irreducible. 
\end{proof}

We wonder whether Theorem \ref{thm:min-prox-iff-alg-irred} could be used to establish the existence of non-trivial minimal proximal flows for non-FC-hypercentral groups, thereby giving a new proof of the main result in \cite{FTV2019}. We will return to the question of the existence of non-trivial minimal proximal flows in later sections.

\section{Highly proximal flows and injectivity} \label{sec:injectivity}

Zucker recently showed in \cite{Zuc2020} that if $G$ is a Polish group, then the universal minimal $G$-flow $\mb G$, the universal minimal proximal $G$-flow $\gb G$ and the universal strongly proximal $G$-flow $\fb G$ are all maximally highly proximal (see Section \ref{sec:max-highly-proximal}). If $G$ is discrete, then it turns out that a $G$-flow $X$ is maximally highly proximal if and only if the C*-algebra $\rC(X)$ is injective in the category of unital commutative C*-algebras.

Injectivity is an important and well-studied property of C*-algebras. However, a key observation in \cite{KK2017} is that the C*-algebra $\rC(\fb G)$ is actually injective in a much stronger sense. Specifically, it is injective in the category of  function systems equipped with a $G$-action. Here, a function system refers to a closed unital self-adjoint subspace of a unital commutative C*-algebra.

In this section we will consider a property of $G$-flows that corresponds to this stronger notion of injectivity. We will show that, in addition to the universal minimal strongly proximal $G$-flow $\fb G$, the universal pointed $G$-flow (i.e. the Stone--\v{C}ech compactification) $\beta G$, the universal minimal $G$-flow $\mb G$ and the universal minimal proximal $G$-flow $\gb G$ also have this property.

\subsection{Injectivity and essentiality}
Since we will be considering injective objects in two different categories, it will be convenient to give a broad overview of the basic category theory that we will require.

Let $\C$ be a category with objects $\Obj(\C)$ and morphisms $\Mor(\C)$. Let $E \subseteq \Mor(\C)$ be a class of morphisms. We will refer to the morphisms in $E$ as {\em embeddings}, since for the categories that we will consider, $E$ will consist of morphisms that are ``embeddings'' in an appropriate sense.

For example, in the next section we will consider the category of unital commutative C*-algebras with unital *-homomorphisms as morphisms. In this setting, the class $E$ will be the unital *-monomorphism.

An object $I \in \Obj(\C)$ is {\em injective} if for any objects $A,B \in \Obj(\C)$, any embedding $\iota : A \to B$ and any morphism $\phi : A \to I$, there is a morphism $\psi : B \to I$ such that $\psi \circ \iota = \phi$, or equivalently, such that the following diagram commutes.

\begin{center}
\begin{tikzpicture}[ampersand replacement=\&]
  \matrix(m)[matrix of math nodes,row sep=2em,column sep=4.5em,minimum width=1em]
  {B  \&   \\
   A  \& I \\};
  \path[{Hooks[right]}-stealth] (m-2-1) edge node [left] {$\iota$} (m-1-1);
  \path[-stealth] (m-2-1) edge node [below] {$\phi$} (m-2-2);
  \path[-stealth] (m-1-1) edge [dashed] node [above] {$~~~~~~~~\exists~$$\psi$} (m-2-2);
\end{tikzpicture}
\end{center}

For an object $A \in \Obj(\C)$, we will be interested in finding an injective object $I \in \Obj(\C)$ along with an embedding $\iota : A \to I$ that is {\em minimal} in the sense that if $J \in \Obj(\C)$ is injective and $\kappa : A \to J$ is an embedding, then there is a surjective morphism $\phi : J \to I$ such that $\phi \circ \kappa = \iota$. A pair $(A,\iota)$ with this property is an {\em injective hull} for $A$. For the categories that we will consider, it is a non-trivial result that every object has an injective hull that is unique up to isomorphism.

An embedding $\iota : A \to B$ is {\em essential} if for every object $C \in \Obj(\C)$ and every morphism $\phi : B \to C$, $\phi$ is an embedding whenever $\phi \circ \iota$ is an embedding, or equivalently, whenever the following diagram commutes.

\begin{center}
\begin{tikzpicture}[ampersand replacement=\&]
  \matrix(m)[matrix of math nodes,row sep=2em,column sep=4.5em,minimum width=1em]
  {B  \&   \\
   A  \& C \\};
  \path[{Hooks[right]}-stealth] (m-2-1) edge node [left] {$\iota$} (m-1-1);
  \path[{Hooks[right]}-stealth] (m-2-1) edge node [below] {$\in E$} (m-2-2);
  \path[{Hooks[right]}-stealth] (m-1-1) edge node [above] {$~~~\phi$}   (m-2-2);
\end{tikzpicture}
\end{center}

There is typically a close relationship between injective objects and essential embeddings. For the categories that we will consider here, an object $I \in \Obj(\C)$ is injective if and only if it is {\em maximal essential}, meaning that for every object $A \in \Obj(\C)$, an essential embedding $\iota : I \to A$ is an isomorphism.

\subsection{Maximally highly proximal flows}
\label{sec:max-highly-proximal}

The notion of a maximally highly proximal flow was introduced by Auslander and Glasner \cite{AG1977} (see also \cite{Zuc2020}). Let $G$ be a discrete group and let $X$ and $Y$ be $G$-flows. An extension $\phi : Y \to X$ is {\em highly proximal} if for every non-empty open subset $U \subseteq Y$ there is a point $x \in X$ such that $\phi^{-1}(x) \subseteq U$. The extension $Y \to X$ is {\em universally highly proximal} if whenever $Z$ is a $G$-flow and $\psi : Z \to X$ is a highly proximal extension, then there is an extension $\rho : Y \to Z$ such that $\phi = \psi \circ \rho$.  The universal highly proximal extension is unique up to isomorphism.  The $G$-flow $X$ is called {\em maximally highly proximal} if the identity map $\id_X$ is the universal highly proximal extension.

Recall that the category with $G$-flows as objects and $G$-equivariant continuous maps as morphisms is dually equivalent to the category with unital commutative $G$-C*-algebras as objects and $G$-equivariant unital *-homomorphisms as morphisms. An embedding in this latter category is a unital *-monomorphism. 

A pair of $G$-flows $X$ and $Y$ correspond to the $G$-C*-algebras $\rC(X)$ and $\rC(Y)$, and an extension $\alpha : Y \to X$ corresponds to an embedding $\pi : \rC(X) \to \rC(Y)$ via the formula $f \circ \alpha(y) = \pi(f)(y)$ for $f \in \rC(X)$ and $y \in Y$.

In order to identify the C*-algebras corresponding to maximally highly proximal flows, we will work within the category of unital commutative C*-algebras. Before stating the results, we briefly review some definitions.

For compact Hausdorff spaces $X$ and $Y$, a *-homomorphism $\pi : \rC(X) \to \rC(Y)$ is an {\em embedding} if it is a *-monomorphism. It follows from Zorn's lemma that every unital commutative C*-algebra $\rC(X)$ admits an essential embedding $\pi : \rC(X) \to \rC(Y)$ which is maximal in the following sense: whenever $\rho:\rC(Y) \to \rC(Z)$ is an embedding such that $\rho \circ \pi$ is essential, then $\rho$ is a $*$-isomorphism.

Gleason \cite{Gle1958} (see also \cite{HP2011}) showed that the unital commutative C*-algebra $\rC(Y)$ is injective if and only if the space $Y$ is extremally disconnected. Gonshor \cite{Gon1968} showed that this is equivalent to $\rC(Y)$ being maximal essential. Moreover, Gleason showed that every unital commutative C*-algebra $\rC(X)$ has an injective hull $\rC(Y)$ that is unique up to isomorphism. The space $Y$ is often referred to as the {\em Gleason cover} of $X$. 

The next result from \cite{Gon1968} characterizes essential embeddings of unital commutative C*-algebras in terms of the corresponding compact Hausdorff spaces.

\begin{thm} \label{thm:injectivity-equivalences}
Let $X$ and $Y$ be compact Hausdorff spaces and let $\pi : \rC(X) \to \rC(Y)$ be an embedding. Let $\phi : Y \to X$ denote the corresponding continuous map. The following are equivalent:
\begin{enumerate}
\item The embedding $\pi$ is essential.
\item For every nonzero ideal $J$ in $\rC(Y)$, $J \cap \pi(\rC(X)) \ne 0$. 
\item For every proper closed subset $C \subseteq Y$, $\phi(C) \ne X$.
\end{enumerate}
\end{thm}

Gonshor \cite{Gon1970} showed that the Gleason cover of a compact Hausdorff space $X$ is isomorphic to the Stone space of the Boolean algebra of regular open subsets of $X$. Recall that an open subset $U \subseteq X$ is {\em regular} if it is equal to the interior of its closure. The regular open subsets of $X$ form a Boolean algebra with respect to the following operations: 
\begin{itemize}
\item $U \land V$ is $U \cap V$,
\item $U \lor V$ is the interior of the closure of $U \cup V$ and
\item $\neg U$ is the interior of $U^c$.
\end{itemize}

\begin{thm} \label{thm:maximally-highly-proximal-injective-hull}
Let $G$ be a discrete group and let $X$ and $Y$ be minimal $G$-flows. An extension $\phi : Y \to X$ is highly proximal if and only if the corresponding embedding $\pi : \rC(X) \to \rC(Y)$ is essential. Hence $Y$ is maximally highly proximal if and only if $\rC(Y)$ is the injective hull of $\rC(X)$. In this case, $Y$ is isomorphic to the Stone space of the Boolean algebra of regular open subsets of $X$.
\end{thm}

\begin{proof}
Suppose that $\phi$ is highly proximal. Let $C \subseteq Y$ be a proper closed subset. Then $V = Y \setminus C$ is a non-empty open subset, so there is $x \in X$ such that $\phi^{-1}(x) \subseteq V$. Hence $x \notin \phi(C)$. Therefore, by (3) of Theorem \ref{thm:injectivity-equivalences}, $\pi$ is essential.

Conversely, suppose that $\pi$ is essential. Fix a non-empty open subset $V \subseteq Y$. Let $C = Y \setminus V$, so that $C$ is a proper closed subset. Then by (3) of Theorem \ref{thm:injectivity-equivalences}, $\phi(C) \ne X$. Hence there is $x \in X$ with $\phi^{-1}(x) \subseteq V$. Therefore, $\phi$ is highly proximal.

The equivalence between $Y$ being maximally highly proximal and $\rC(Y)$ being the injective hull of $\rC(X)$ is now implied by the characterization of the injective hull of $\rC(X)$ as maximal essential. The final statement describing $Y$ when $\rC(Y)$ is injective is from \cite{Gon1970}.
\end{proof}

\begin{example} \label{ex:maximally-highly-proximal}
Let $G$ be a discrete group and let $X$ and $Y$ be $G$-flows such that $Y$ is maximally highly proximal and let $\alpha : Y \to X$ be an extension. By Theorem \ref{thm:maximally-highly-proximal-injective-hull}, the C*-algebra $\rC(Y)$ is injective in the category of unital commutative C*-algebras. However, it is not necessarily injective in the category of unital commutative $G$-C*-algebras.

To see this, suppose that $G$ is non-trivial and let $X$ denote the trivial $G$-flow. Then $X$ is maximally highly proximal, so by Theorem \ref{thm:maximally-highly-proximal-injective-hull}, $\bC \cong \rC(X)$ is injective (note that this also follows from the Hahn-Banach theorem).  However, $\bC$ is not $G$-injective in the category of unital commutative $G$-C*-algebras. Otherwise the identity map $\id_{\bC}$ would extend to an equivariant *-homomorphism $\pi: \ell^\infty(G) \to \bC$.  Since $\ell^\infty(G) \cong \rC(\beta G)$, such an extension would be of the form $\pi = \delta_x$ for some $x \in \beta G$, where $\delta_x$ denotes the Dirac measure corresponding to $x$. However, by \cite{Ell1960} $\beta G$ has no $G$-fixed point, so this is impossible. Therefore, $\bC$ is not injective in the category of unital commutative $G$-C*-algebras.

\end{example}

\subsection{Maximally affinely highly proximal flows}

In this section we will work with {\em function systems}, which are closed unital self-adjoint subspaces of unital commutative C*-algebras. For a reference on the theory of function systems, we refer the reader to the recent paper of Paulsen and Tomforde \cite{PT2009}.

Let $\R$ and $\S$ be function systems. A unital linear map $\phi : \R \to \S$ is an {\em order homomorphism} if $\phi(f) \geq 0$ for all $f \in \R$ with $f \geq 0$. It is an {\em order isomorphism} if it has an inverse that is also an order homomorphism.

If $K$ is a compact convex set then the space $\rA(K)$ of continuous affine functions on $K$ is a function system. Moreover, Kadison's \cite{Kad1951} representation theorem implies that every function system is isomorphic to a function system of this form. Specifically, let $\R$ be a function system and let $K$ denote the compact convex set of states on $\R$, i.e. the unital order homomorphisms from $\R$ to $\bC$ equipped with the weak* topology. Then the map $\phi : \R \to \rA(K)$ defined by $\phi(f)(\alpha) = \alpha(f)$ for $f \in \R$ and $\alpha \in K$ is a unital order isomorphism.

The category with compact convex sets as objects and continuous affine maps as morphisms is dual to the category with function systems as objects and order homomorphisms as morphisms. An embedding in the latter category is an order monomorphism.

Let $G$ be a discrete group. An action of $G$ on a function system $\R$ is a group homomorphism from $G$ into the group of order automorphisms of $\R$. We will refer to $\R$ as a {\em $G$-function system}.

It follows from the duality between the category of compact convex sets and the category of function systems that the category with affine $G$-flows as objects and equivariant continuous affine maps as morphisms is dual to the category with $G$-function systems as objects and equivariant order homomorphisms as morphisms. An embedding in the latter category is an equivariant order monomorphism.

Note that every unital $G$-C*-algebra is isomorphic (i.e. equivariantly order isomorphic to) a $G$-function system. If $X$ and $Y$ are $G$-flows and $\alpha : Y \to X$ is an extension then the corresponding equivariant unital *-monomorphism $\pi : \rC(X) \to \rC(Y)$ is an embedding in the category of $G$-function systems.

A pair of affine $G$-flows $K$ and $L$ correspond to the $G$-function systems $\rA(K)$ and $\rA(L)$, and an affine extension $\alpha : L \to K$ corresponds to an embedding $\phi : \rA(K) \to \rA(L)$ via the formula $f \circ \alpha(y) = \phi(f)(y)$ for $f \in \rA(K)$ and $y \in L$.

Hamana \cite{Ham1985} showed that a $G$-function system is injective if and only if it is maximal essential, and used this to show that every $G$-function system has an injective envelope that is unique up to isomorphism. Hamana further observed that an injective $G$-function system is isomorphic (as a $G$-function system) to a $G$-C*-algebra. However, as the next exmaple shows, a unital commutative $G$-C*-algebra that is injective in the category of unital commutative C*-algebras may not be injective in the category of $G$-function systems.

\begin{example} \label{ex:furstenberg-boundary}
Let $G$ be a discrete group and consider the complex numbers $\bC$ equipped with the trivial $G$-action. Note that $\bC$ is a unital commutative $G$-C*-algebra that is injective in the category of unital commutative C*-algebras. In particular, it is also a function system.

It was shown in \cite{KK2017} that the injective hull of $\bC$ in the category of $G$-function systems is the C*-algebra $\rC(\fb G)$, where $\fb G$ denotes the universal minimal strongly proximal $G$-flow. In particular, the embedding $\bC \subseteq \rC(\fb G)$ is essential.

Furstenberg \cite{Fur1973} showed that $\fb G$ is trivial if and only if $G$ is amenable. This can also be seen using the ideas introduced in this section. Choose $x \in \fb G$ and let $\pi_x : \rC(\fb G) \to \ell^\infty(G)$ denote the equivariant unital *-homomorphism defined by $\pi_x(f)(g) = f(gx)$ for $f \in \rC(\fb G)$ and $g \in G$. Then since $\fb G$ is minimal, $\pi_x$ is an embedding.

Suppose that $G$ is amenable. Then there is an invariant state on $\ell^\infty(G)$. Equivalently, there is an equivariant order homomorphism $\phi : \ell^\infty(G) \to \bC$. Since the embedding $\bC \subseteq \rC(\fb G)$ is essential, it follows that $\phi|_{\pi_x(\rC(\fb G))}$ is also an embedding. Hence $\rC(\fb G) \cong \bC$ and $\fb G$ is trivial, so $\bC$ is an injective $G$-function system.

Conversely, if $G$ is non-amenable, then since $\rC(\fb G)$ is injective in the category of $G$-function systems, there is an equivariant order homomorphism $\psi : \ell^\infty(G) \to \rC(\fb G)$ (see Proposition \ref{prop:expectations} below). However, since $G$ is non-amenable, there is no invariant state on $\ell^\infty(G)$. Therefore, the range of $\psi$ cannot be $\bC$, implying that $\rC(\fb G) \not \cong \bC$. Hence $\fb G$ is non-trivial, so $\bC$ is not an injective $G$-function system.
\end{example}

The next result characterizes essential embeddings of $G$-function systems in terms of the corresponding affine $G$-flows. This is a special case of a more general result from \cite{KS2019}*{Theorem 7.4}.

\begin{prop} \label{prop:G-essential-equivalences}
Let $G$ be a discrete group and let $K$ and $M$ be affine $G$-flows. Let $\phi : \rA(K) \to \rA(M)$ be an embedding and let $\alpha : M \to K$ denote the corresponding equivariant continuous affine map. The following are equivalent:
\begin{enumerate}
\item The embedding $\phi$ is essential.
\item The extension $\alpha : M \to K$ is $K$-irreducible in the sense of \cite{Gla1975}, i.e. for every proper affine subflow $L \subseteq M$, $\alpha(L) \ne K$.
\end{enumerate}
\end{prop}

\begin{proof}
(1) $\Rightarrow$ (2) Let $L \subseteq M$ be an affine subflow such that $\alpha(L) = K$. Let $\psi : \rA(M) \to \rA(L)$ denote the restriction map. Then $\psi$ is an equivariant order homomorphism and $\psi \circ \phi$ is isometric, and therefore is an equivariant order isomorphism. The essentiality of $\phi$ implies that $\psi$ is an order isomorphism. Hence $L = M$.

(2) $\Rightarrow$ (1) Let $N$ be an affine $G$-flow and let $\psi : \rA(M) \to \rA(N)$ be an equivariant order homomorphism such that $\psi \circ \phi$ is an isomorphism. Let $\beta : N \to M$ denote the corresponding equivariant continuous affine map and let $L = \beta(N)$. Then $L$ is an affine subflow and $\alpha(L) = K$. Hence $L = M$, so $\psi$ is an isomorphism. Therefore, $\phi$ is essential.
\end{proof}

\begin{defn}
  \label{def:affinely-highly-proximal}
Let $G$ be a discrete group.
\begin{enumerate}
\item Let $K$ and $M$ be affine $G$-flows. We will say that an extension $\alpha : M \to K$ is {\em affinely highly proximal} if for every proper affine $G$-subflow $L \subseteq M$, there is a point $x \in K$ such that $\alpha^{-1}(Gx) \subseteq M \setminus L$.
\item Let $X$ and $Y$ be $G$-flows. We will say that an extension $\alpha : Y \to X$ is {\em affinely highly proximal} if the corresponding affine extension $\tilde{\alpha} : \rP(Y) \to \rP(X)$ is affinely highly proximal.
\end{enumerate}
{\em Universally affinely highly proximal} extensions and {\em maximally affinely highly proximal} affine $G$-flows are defined as they were for $G$-flows in Section \ref{sec:max-highly-proximal}. We will say that a $G$-flow $Y$ is {\em maximally affinely highly proximal} if the corresponding affine $G$-flow $\rP(Y)$ is maximally affinely highly proximal.
\end{defn}

\begin{rem}
Note that there is no ambiguity in Definition \ref{def:affinely-highly-proximal}, since if $K$ and $M$ are affine $G$-flows and $\alpha : M \to K$ is an affinely highly proximal extension, then the corresponding extension $\tilde{\alpha} : \rP(M) \to \rP(K)$ is also affinely highly proximal. To see this, let $L \subseteq \rP(M)$ be an affine $G$-subflow. Let $\beta_K : \rP(K) \to K$ and $\beta_M : \rP(M) \to M$ denote the barycenter maps. Then $\beta_K$ and $\beta_M$ are equivariant and continuous, and $\beta_K \circ \tilde{\alpha} = \alpha \circ \beta_M$. 

It follows from above that $\beta(L) \subseteq M$ is an affine $G$-subflow, so by assumption there is $x \in K$ such that $\alpha^{-1}(Gx) \subseteq M \setminus \beta_M(L)$. Applying $\beta_M^{-1}$ to both sides gives
\begin{multline*}
\tilde{\alpha}^{-1}(Gx) \subseteq \tilde(\alpha)^{-1} \circ \beta_K^{-1}(Gx) = \beta_M^{-1} \circ \alpha^{-1}(Gx)
\subseteq \beta_M^{-1}(M \setminus \beta_M(L)) \\
= \rP(M) \setminus \beta_M^{-1}(\beta_M(L)) \subseteq \rP(M) \setminus L.
\end{multline*}
Hence $\tilde{\alpha}$ is affinely highly proximal.
\end{rem}

\begin{thm} \label{thm:affinely-highly-proximal-essential}
Let $G$ be a discrete group and let $K$ and $M$ be affine $G$-flows. Let $\alpha : M \to K$ be an extension and let $\phi : \rA(K) \to \rA(M)$ denote the corresponding embedding. Then $\alpha$ is affinely highly proximal if and only if $\phi$ is essential. Hence $M$ is maximally affinely highly proximal if and only if $\rA(M)$ is the injective hull of $\rA(K)$. 
\end{thm}

\begin{proof}
Suppose that $\alpha$ is affinely highly proximal. Let $L \subseteq M$ be a proper affine subflow. Then there is $x \in K$ such that $\alpha^{-1}(Gx) \subseteq M \setminus L$. In particular, $\alpha(L)$ is proper. Hence by (2) of Proposition \ref{prop:G-essential-equivalences}, $\phi$ is essential.

Conversely, suppose that $\phi$ is essential. Let $L \subseteq M$ be a proper affine subflow. Then by (2) of Proposition \ref{prop:G-essential-equivalences}, $\alpha(L) \ne K$. Since $K \setminus \alpha(L)$ is invariant, there is $x \in K$ such that $Gx \in K \setminus \alpha(L)$. Hence $\alpha^{-1}(Gx) \subseteq M \setminus L$. Therefore, $\alpha$ is affinely highly proximal.

The final statement follows from the characterization of the injective hull of $\rA(K)$ as maximal essential.
\end{proof}

Since every unital commutative $G$-C*-algebra is also (equivariantly order isomorphic to) a $G$-function system, and since an equivariant *-monomorphism is an embedding in the category of $G$-function systems, Theorem \ref{thm:affinely-highly-proximal-essential} immediately implies the following result.

\begin{cor} \label{cor:affinely-highly-proximal-essential}
Let $G$ be a discrete group and let $X$ and $Y$ be $G$-flows. Let $\alpha : Y \to X$ be an extension and let $\pi : \rC(X) \to \rC(Y)$ denote the corresponding equivariant unital *-monomorphism. Then $\alpha$ is affinely highly proximal if and only if $\pi$ is an essential embedding in the category of $G$-function systems. Hence $Y$ is maximally affinely highly proximal if and only if $\rC(Y)$ is the injective hull of $\rC(X)$ in the category of $G$-function systems.
\end{cor}

\begin{rem}
The notion of an affinely highly proximal extension is related to Glasner's \cite{Gla1975} notion of a strongly proximal extension. Let $X$ and $Y$ be $G$-flows. An extension $\alpha : Y \to X$ is strongly proximal if the corresponding extension $\tilde{\alpha} : \rP(Y) \to \rP(X)$ has the property that whenever $\mu \in \rP(Y)$ satisfies $\tilde{\alpha}(\mu) \in X$, then there is a net $(g_i)$ in $G$ such that $\lim g_i \mu \in Y$. 

If $X$ is minimal and $\alpha$ is affinely highly proximal, then by Proposition \ref{prop:G-essential-equivalences} and Corollary \ref{cor:affinely-highly-proximal-essential}, $\tilde{\alpha}$ is $\rP(X)$-irreducible. Hence by \cite{Gla1975}*{Theorem 2.1}, $\alpha$ is strongly proximal. On the other hand, if $Y$ is minimal and $\alpha$ is strongly proximal, then by \cite{Gla1975}*{Proposition 2.2}, $\tilde{\alpha}$ is $\rP(X)$-irreducible. Hence by Proposition \ref{prop:G-essential-equivalences} and Corollary \ref{cor:affinely-highly-proximal-essential}, $\alpha$ is affinely highly proximal.
\end{rem}

If $X$ is a maximally affinely highly proximal $G$-flow, then Corollary \ref{cor:affinely-highly-proximal-essential} implies that $\rC(X)$ is injective as a $G$-function system. Hamana observed in \cite{Ham1985}*{Remark 2.3} that this implies that $\rC(X)$ is also injective in the category of commutative C*-algebras. Hence by Theorem \ref{thm:maximally-highly-proximal-injective-hull}, $X$ is maximally highly proximal. However, Example \ref{ex:furstenberg-boundary} shows that the converse is not true in general.

The next result implies the existence of relatively invariant measures in the sense of Glasner \cite{Gla1975} for maximally affinely highly proximal affine $G$-flows. However, since not every $G$-function system is isomorphic to a $G$-C*-algebra, the present setting is slightly more general than the setting considered by Glasner. In the theory of operator algebras, relatively invariant measures are typically referred to as conditional expectations. 

\begin{prop} \label{prop:expectations}
Let $G$ be a discrete group and let $K$ and $L$ be affine $G$-flows. Let $\alpha : L \to K$ be an extension with corresponding embedding $\phi : \rA(K) \to \rA(L)$.
\begin{enumerate}
\item If $K$ is maximally affinely highly proximal then there is an equivariant order homomorphism $\psi : \rA(L) \to \rA(K)$ such that $\psi \circ \phi$ is the identity map on $\rA(K)$. 
\item If $L$ is maximally affinely highly proximal and there is an equivariant order homomorphism $\psi : \rA(L) \to \rA(K)$ such that $\psi \circ \phi = \id_{\rA(K)}$, then $K$ is also maximally affinely highly proximal. 
\end{enumerate}
\end{prop}

\begin{proof}
(1) Since $K$ is maximally affinely highly proximal, Theorem \ref{thm:affinely-highly-proximal-essential} implies that $\rA(K)$ is $G$-injective. Hence the identity map on $\rA(K)$ extends to an equivariant order homomorphism $\psi : \rA(L) \to \rA(K)$, that is $\psi \circ \phi = \id_{\rA(K)}$. 

(2) Let $\R$ and $\S$ be function systems and let $\rho : \R \to \S$ be an embedding. Let $\sigma : \R \to \rA(K)$ be an equivariant order homomorphism. Since $L$ is maximally affinely highly proximal, Theorem \ref{thm:affinely-highly-proximal-essential} implies that $\rA(L)$ is $G$-injective. Hence there is an equivariant order homomorphism $\tau : \S \to \rA(L)$ such that $\phi \circ \sigma = \tau \circ \rho$. Then $\psi \circ \tau$ extends $\sigma$, since $\psi \circ \tau \circ \rho = \psi \circ \varphi \circ \sigma = \sigma$.
Hence $\rA(K)$ is $G$-injective.
\end{proof}

\begin{thm} \label{thm:boundaries-injective}
Let $G$ be a discrete group. Then the Stone--\v{C}ech compactification $\beta G$, the universal minimal $G$-flow $\mb G$, the universal strongly proximal $G$-flow $\fb G$ and the universal minimal proximal $G$-flow $\gb G$ are all maximally affinely highly proximal. Hence the C*-algebras $\rC(\beta G)$, $\rC(\mb G)$, $\rC(\fb G)$ and $\rC(\gb G)$ are all $G$-injective in the category of $G$-function systems.
\end{thm}

\begin{proof}
We will show that the C*-algebras $\rC(\beta G)$, $\rC(\mb G)$, $\rC(\fb G)$ and $\rC(\gb G)$ are all injective in the category of $G$-function systems. The fact that the corresponding $G$-flows are maximally affinely highly proximal will then follow from Theorem \ref{thm:affinely-highly-proximal-essential}. 

The injectivity of $\rC(\beta G) \cong \ell^\infty(G)$ as a $G$-function system was shown in \cite{Ham1985}*{}, and the injectivity of $\rC(\fb G)$ as a $G$-function system was shown in \cite{KK2017}*{}. 

To see that $\rC(\mb G)$ is injective as a $G$-function system, let $L$ be a minimal ideal of $\beta G$ and identify $L$ with $\mb G$. Let $u \in L$ be an idempotent. Since $\mb G$ is minimal, the map $\beta G \to L: x \mapsto xu$ is an extension. Let $\pi_u : \rC(L) \to \rC(\beta G)$ denote the corresponding equivariant unital *-monomorphism. Then $\pi_u$ is an embedding in the category of $G$-function systems.

Let $\rho : \rC(\beta G) \to \rC(L)$ denote the equivariant surjective unital *-homomorphism corresponding to the inclusion $L \subseteq \beta G$. In other words, $\rho$ is the restriction map. Then for $f \in \rC(L)$ and $x \in L$, $\rho \circ \pi_u(f)(x) = f(xu) = f(x)$. Hence $\phi \circ \pi_u = \id_{\rC(L)}$. Therefore, it follows from the injectivity of $\rC(\beta G)$ as a $G$-function system and (2) of Proposition \ref{prop:expectations} that $\rC(L)$ is also injective as a $G$-function system.

Finally, to see that $\rC(\gb G)$ is injective as a $G$-function system, let $Y$ be a maximally affinely highly proximal $G$-flow and let $\alpha : Y \to \gb G$ be an affinely highly proximal extension. Then $\rC(Y)$ is the injective hull of $\rC(\gb G)$. We will show that $\alpha$ is an isomorphism. 

Let $K = \rP(\gb G)$ and let $M = \rP(Y)$. Then $K$ and $M$ are affine $G$-flows and there is an extension $\tilde{\alpha} : M \to K$ extending $\alpha$. By Theorem \ref{thm:affinely-highly-proximal-essential}, $\tilde{\alpha}$ is affinely highly proximal.

By Theorem \ref{thm:min-prox-iff-alg-irred}, $K$ is algebraically irreducible. We claim that $M$ is also algebraically irreducible. To see this, let $L \subseteq M$ be an affine subflow such that $L \cap \operatorname{conv}(Y) \ne \emptyset$. Then $\tilde{\alpha}(L) \subseteq K$ is an affine subflow and $\tilde{\alpha}(L) \cap \operatorname{conv}(\gb G) \ne \emptyset$. Hence by the algebraic irreducibility of $K$, $\tilde{\alpha}(L) = K$. Since $\tilde{\alpha}$ is affinely highly proximal, $L = M$. Therefore, $M$ is algebraically irreducible.

By Proposition \ref{prop:alg-irred-unique-min-subflow}, $Y$ is minimal and proximal. Since $\gb G$ is the universal minimal proximal $G$-flow, it follows that $\alpha$ is an isomorphism.
\end{proof}

An important consequence of the fact that the $G$-flows $\mb G$, $\gb G$ and $\fb G$ are maximally highly proximal, pointed out in  \cite{GTWZ2019}, is that they are extremally disconnected. We will make use of this throughout the remainder of this paper.

\section{Realizations of universal flows} \label{sec:realizations}

In this section we will show that the universal minimal proximal flow and the universal minimal strongly proximal flow of a discrete group are isomorphic to Stone spaces of translation invariant Boolean algebras of subsets of the group that are maximal with respect to the property that every non-empty element has one of the properties considered in Section~\ref{sec:higher-order-syndeticity}. For the universal minimal proximal flow, the non-zero elements are completely syndetic, while for the universal minimal strongly proximal flow, the elements are strongly completely syndetic. 

These results are analogues of a result of Balcar and Franek \cite{BF1997}. They showed that the universal minimal flow is isomorphic to the Stone space of a translation invariant Boolean algebra of subsets of the group that is maximal with respect to the property that every non-empty element is syndetic.

\subsection{Totally disconnected point transitive flows}

Let $G$ be a discrete group and let $2^G$ denote the Boolean algebra of subsets of $G$. We will say that a Boolean subalgebra $\fB \subseteq 2^G$ is {\em translation invariant} if $gA \in \fB$ for every $A \in \fB$ and $g \in G$.

Recall that the Stone space $X$ of $\fB$ is the space of ultrafilters on $\fB$ equipped with a compact totally disconnected Hausdorff topology. For $A \in \fB$, the set $\{x \in X : A \in x\}$ is a clopen subset of $X$ and the family of all such clopen subsets corresponding to elements in $\fB$ is a basis for the topology on $X$. Since $\fB$ is translation invariant, $X$ is a $G$-flow with respect to the corresponding action by left translation. Moreover, identifying $e \in G$ with the corresponding principal ultrafilter in $\beta G$, the point $e \cap \fB \in X$ is transitive. 

\begin{prop} \label{prop:correspondence-boolean-subalgebras}
Let $G$ be a discrete group. The Stone space of a translation invariant Boolean subalgebra of $2^G$ is a totally disconnected point transitive $G$-flow. Conversely, every totally disconnected point transitive $G$-flow is isomorphic to the Stone space of a translation invariant Boolean subalgebra of $2^G$.
\end{prop}

\begin{proof}
It follows from the introduction to this section that the Stone space of a translation invariant Boolean subalgebra of $2^G$ is a totally disconnected point transitive $G$-flow. For the converse, suppose that $X$ is a totally disconnected $G$-flow with a transitive point $x \in X$. Since $X$ is totally disconnected, there is a Boolean algebra $\fC$ of clopen subsets of $X$ that forms a basis for the topology on $X$. For $U \in \fC$, let $U_x = \{g \in G : gx \in U\}$ denote the corresponding return set. Then since $x$ is a transitive point, the equivariant Boolean algebra homomorphism $\fC \to 2^G : U \to U_x$ is an embedding of $\fC$ into $2^G$.
\end{proof}

From \cite{Zuc2020}, we know that the universal minimal proximal $G$-flow $\gb G$ and the universal minimal strongly proximal $G$-flow $\fb G$ are extremally disconnected (and in particular totally disconnected). Hence by Proposition \ref{prop:correspondence-boolean-subalgebras}, they are isomorphic to the Stone spaces of translation invariant Boolean subalgebras of $2^G$. In the next two sections, we will describe these subalgebras in terms of the subsets considered in Section~\ref{sec:higher-order-syndeticity}.

\subsection{Universal minimal proximal flow}

\begin{lem} \label{lem:minimal-proximal-completely-syndetic}
Let $G$ be a discrete group and let $X$ be a minimal proximal $G$-flow. Then for every non-empty open subset $U \subseteq X$ and every point $x \in X$, the return set $U_x = \{g \in G : gx \in U \}$ is completely syndetic.
\end{lem}

\begin{proof}
Let $K \subseteq \beta G$ be a finite subset. Then $Kx \subseteq X$ is also finite. Since $X$ is minimal and proximal there is $g \in G$ such that $gKx \subseteq U$. This implies that $gK \subseteq \ol{U_x}$. Hence by Proposition \ref{prop:n-syndetic-thickness}, $U_x$ is completely syndetic.
\end{proof}

\begin{prop} \label{prop:minimal-proximal-completely-syndetic}
Let $G$ be a discrete group and let $\fB \subseteq 2^G$ be a translation invariant Boolean subalgebra. Let $X$ denote the Stone space of $\fB$. Then $X$ is a $G$-flow and $X$ is minimal and proximal if and only if every non-empty element in $\fB$ is completely syndetic.
\end{prop}

\begin{proof}
Suppose that $X$ is minimal and proximal. Identify $e \in G$ with the corresponding principal ultrafilter in $\beta G$ and let $x = e \cap \fB$. Then for $A \in \fB$, if $U = \{y \in X : A \in y \} \subseteq X$ denotes the corresponding clopen subset, $A = U_x$, where $U_x = \{g \in G : gx \in U\}$ denotes the return set of $U$ at $x$. It follows from Lemma \ref{lem:minimal-proximal-completely-syndetic} that if $A \ne \emptyset$, then $A$ is completely syndetic.

Conversely, suppose that every non-empty element in $\fB$ is completely syndetic. Choose non-empty $A \in \fB$ and let $U = \{x \in X : A \in x \} \subseteq X$ denote the corresponding clopen subset. For points $x,y \in X$, let $u,v \in \beta G$ be ultrafilters extending $x$ and $y$ respectively. Then by Proposition \ref{prop:n-syndetic-thickness} there is $g \in G$ such that $A \in gu,gv$. Hence $A \in gx,gy$, so $gx,gy \in U$. Therefore, $X$ is minimal and proximal. 
\end{proof}

\begin{rem}
It follows from the proof of Proposition \ref{prop:minimal-proximal-completely-syndetic} that in order for $X$ to be minimal and proximal, it suffices for every non-empty element in $\fB$ to be $2$-syndetic.
\end{rem}

\begin{thm} \label{thm:minimal-proximal-completely-syndetic}
Let $G$ be a discrete group. Then the universal minimal proximal $G$-flow $\gb G$ is isomorphic to the Stone space of any translation invariant Boolean subalgebra of $2^G$ that is maximal with respect to the property that every non-empty element is completely syndetic.
\end{thm}

\begin{proof}
We first show that $\gb G$ is isomorphic to the Stone space of a translation invariant Boolean subalgebra of $2^G$ that is maximal with respect to the property that every non-empty element is completely syndetic.

By \cite{Zuc2020}*{}, $\gb G$ is extremally disconnected, so Proposition \ref{prop:correspondence-boolean-subalgebras} implies that there is a translation invariant Boolean subalgebra $\fB \subseteq 2^G$ such that the Stone space of $\fB$ is isomorphic to $\gb G$. By Proposition \ref{prop:minimal-proximal-completely-syndetic}, every non-empty element in $\fB$ is completely syndetic. It remains to show that $\fB$ is maximal with respect to this property. 

Let $\fC \subseteq 2^G$ be a translation invariant Boolean subalgebra that is maximal with respect to the property that $\fB \subseteq \fC$ and every non-empty element is completely syndetic. Then by Proposition \ref{prop:minimal-proximal-completely-syndetic}, the Stone space $Y$ of $\fC$ is a minimal proximal $G$-flow. The inclusion $\fB \subseteq \fC$ induces an extension $\alpha : Y \to \gb G$. By the universality of $\gb G$, $\alpha$ must be an isomorphism. Therefore, $\fB = \fC$. 

Conversely, let $\fC \subseteq 2^G$ be a translation invariant Boolean subalgebra that is maximal with respect to the property that every non-empty element in $\fC$ is completely syndetic. By Proposition \ref{prop:minimal-proximal-completely-syndetic}, the Stone space $Y$ of $\fC$ is a minimal proximal $G$-flow. Hence by the universality of $\gb G$, there is an extension $\alpha : \gb G \to Y$.

Identify $e \in G$ with the corresponding principal ultrafilter in $\beta G$ and let $y = e \cap \fC \in Y$. Choose $z \in \alpha^{-1}(y)$ and let $\pi_z : \rC(\gb G) \to \rC(\beta G) \cong \ell^\infty(G)$ denote the equivariant *-homomorphism defined by $\pi_z(f)(g) = f(gz)$ for $f \in \rC(\gb G)$ and $g \in G$. Then since $\gb G$ is minimal, $\pi_z$ is an embedding. Furthermore, by the choice of $z$, the restriction of $\pi_z$ is the identity map on $\fC$, when identifying elements of $\fC$ with characteristic functions of clopen subsets of $\beta G$. Therefore, the image $\fB \subseteq 2^G$ of the Boolean algebra of projections in $\gb G$ under $\pi_z$ is a translation invariant Boolean subalgebra with Stone space isomorphic to $\gb G$ such that $\fC \subseteq \fB$. By Proposition \ref{prop:minimal-proximal-completely-syndetic}, every non-empty element in $\fB$ is completely syndetic. Hence by the maximality of $\fC$, we conclude that $\fB = \fC$. It follows that $Y$ is isomorphic to $\gb G$.
\end{proof}

\subsection{Universal minimal strongly proximal flow} \label{sec:universal-minimal-strongly-proximal-flow}

\begin{lem} \label{lem:minimal-strongly-proximal-strongly-syndetic}
Let $G$ be a discrete group and let $X$ be a minimal strongly proximal $G$-flow. Then for every non-empty open subset $U \subseteq X$ and every point $x \in X$, the return set $U_x = \{g \in G : gx \in U\}$ is strongly completely syndetic.
\end{lem}

\begin{proof}
The proof is inspired by the proof of \cite{Gla1976}*{Chapter VII, Proposition~2.1}. Fix an open subset $U \subseteq X$. For every probability measure $\mu \in \rP(X)$ there is $g_\mu \in G$ such that $1 - \mu(g_\mu U) < \epsilon$. Let $V_\mu \subseteq \rP(X)$ be an open neighborhood of $\mu$ such that $1 - \nu(g_\mu U) < \epsilon$ for all $\nu \in V_\mu$. By compactness there is $\mu_1,\ldots,\mu_n \in \rP(X)$ such that $V_{\mu_1},\ldots,V_{\mu_n}$ cover $\rP(X)$. Hence letting $F = \{g_{\mu_1}^{-1},\ldots,g_{\mu_n}^{-1}\}$, it follows that for every $\mu \in \rP(X)$ there is $f \in F$ such that $1 - (f\mu)(U) < \epsilon$.

Let $K \subseteq G$ be a finite multiset and let $\lambda = \frac{1}{|K|} \sum_{k \in K} \delta_k \in \rP(\beta G)$. Let $\mu \in \rP(X)$ denote the push-forward of $\lambda$ along the unique factor map $\beta G \to X$ that maps $e$ to $x$.  Then $\mu = \frac{1}{|K|} \sum_{k \in K} \delta_{kx}$.  Writing $A = U_x$, we then have $\lambda(A) = \mu(U)$, and so from above there is $f \in F$ such that $1 - (f\lambda)(A) < \epsilon$. Hence
\[
\epsilon > 1 - (f\lambda)(A) = 1 - \frac{|fK \cap A|}{|K|},
\]
implying $|fK \cap A| \geq (1-\epsilon)|K|$.  Therefore, $A$ is strongly completely syndetic.
\end{proof}

\begin{prop} \label{prop:minimal-strongly-proximal-strongly-syndetic}
Let $G$ be a discrete group and let $\fB \subseteq 2^G$ be a translation invariant Boolean subalgebra. Let $X$ denote the Stone space of $\fB$. Then $X$ is minimal and strongly proximal if and only if every non-empty element in $\fB$ is strongly completely syndetic.
\end{prop}

\begin{proof}
Suppose that $X$ is minimal and strongly proximal. Identify $e \in G$ with the corresponding principal ultrafilter in $\beta G$ and let $x = e \cap \fB$. Then for $A \in \fB$, if $U = \{y \in X : A \in y \} \subseteq X$ denotes the corresponding clopen subset, $A = U_x$, where $U_x = \{g \in G : gx \in U\}$ denotes the return set of $U$ at $x$. It follows from Lemma \ref{lem:minimal-strongly-proximal-strongly-syndetic} that if $A$ is non-empty, then $A$ is strongly completely syndetic.

Conversely, suppose that every non-empty element in $\fB$ is strongly completely syndetic. To show that $X$ is minimal and strongly proximal, it suffices to show that for every probability measure $\mu \in \rP(X)$, every $\epsilon > 0$ and every non-empty clopen subset $U \subseteq X$, there is $g \in G$ such that $(g\mu)(U) > 1 - \epsilon$.

Let $\mu \in \rP(X)$ be a probability measure and choose $\epsilon > 0$. Let $U \subseteq X$ be a non-empty clopen subset. Then there is $A \in \fB$ such that $U = \{x \in X : A \in x \}$. From above $A$ is strongly completely syndetic, so there is a finite subset $F \subseteq G$ such that for every finite multiset $K \subseteq G$ there is $f \in F$ with $|fK \cap A| \geq (1-\epsilon/2)|K|$. 

Let $\nu \in \rP(\beta G)$ be a probability measure such that $\nu|_{\fB} = \mu$. Choose a finitely supported measure $\lambda \in \ell^1(G)$ such that $|\lambda(f^{-1}A) - \nu(f^{-1}A)| < \epsilon/2$ for all $f \in F$. We can assume there is a finite multiset $K \subseteq G$ such that $\lambda = \frac{1}{|K|} \sum_{k \in K} \delta_k$. 

From above there is $f \in F$ such that $|fK \cap A| \geq (1-\epsilon/2)|K|$. Then
\[
(f\lambda)(A) = \frac{|fK \cap A|}{|K|} \geq 1 - \frac{\epsilon}{2}.
\]
Hence $|(f\mu)(U)| \geq |(f\lambda)(A)| - |(f\nu)(A) - (f\lambda)(A)| > 1 - \epsilon$. Therefore, $X$ is minimal and strongly proximal.
\end{proof}

The proof of the next result is essentially the same as the proof of Theorem \ref{thm:minimal-proximal-completely-syndetic}.

\begin{thm} \label{thm:minimal-strongly-proximal-completely-syndetic}
Let $G$ be a discrete group. Then the universal minimal strongly proximal $G$-flow $\fb G$ is isomorphic to the Stone space of any translation invariant Boolean subalgebra of $2^G$ that is maximal with respect to the property that every non-empty element is strongly completely syndetic.
\end{thm}

\subsection{Symmetric higher order syndeticicity}

In this section we will establish a characterization of subsets of discrete groups that generate translation-invariant Boolean algebras with the property that every non-empty element is (completely, strongly completely) syndetic. These subsets play an important role in Section \ref{sec:dense-orbit-sets}.

\begin{defn}
  \label{def:symmetrically-syndetic}
  Let $G$ be a discrete group. We will say that a subset $A \subseteq G$ is {\em symmetrically (completely, strongly completely) syndetic} if it is (completely, strongly completely) syndetic and for any finite subsets $F_1,F_2 \subseteq G$, the set $(\cap_{f_1 \in F_1} f_1 A) \cap (\cap_{f_2 \in F_2} f_1 A^c)$ is either (completely, strongly completely) syndetic or empty.
\end{defn}

Observe that if $A \subseteq G$ is proper, then it is symmetrically (completely, strongly completely) syndetic if and only if $A^c$ is. This is part of the motivation for our terminology, but we will see more motivation in Section \ref{sec:dense-orbit-sets}. Note that the complement of an ordinary (completely, strongly completely) syndetic subset is not necessarily syndetic.

The next result establishes the key property of (completely, strongly completely) symmetrically syndetic subsets. 

\begin{prop} \label{prop:symmetrically-generated-boolean}
Let $G$ be a discrete group and let $A \subseteq G$ be a subset. Then $A$ is symmetrically (completely, strongly completely) syndetic if and only if it is (completely, strongly completely) syndetic and every non-zero element in the translation-invariant Boolean subalgebra of $2^G$ generated by $A$ is (completely, strongly completely) syndetic.
\end{prop}

\begin{proof}
Observe that for finite subsets $F_1,F_2 \subseteq G$, the intersection $(\cap_{f_1 \in F_1} f_1 A) \cap (\cap_{f_2 \in F_2} f_1 A^c)$ is contained in the translation-invariant Boolean subalgebra generated by $A$. Moreover, every non-empty element in the Boolean subalgebra necessarily contains a non-empty subset of this form, and the property of being (completely, strongly completely) syndetic passes to supersets. 
\end{proof}

\begin{prop} \label{prop:symmetric-return-set}
  \label{prop:symmetrically-syndetic-dynamical-characterization}
  Let $G$ be a discrete group. A subset $A \subseteq G$ is symmetrically (completely, strongly completely) syndetic if and only if the universal minimal (proximal, strongly proximal) flow contains an open subset $U$ and a point $x$ such that $A = U_x$. 
\end{prop}
\begin{proof}
We will first consider the syndetic case. Suppose that $A \subseteq G$ is symmetrically syndetic and let $\fB \subseteq 2^G$ denote the $G$-Boolean algebra generated by $A$. Then by Proposition \ref{prop:symmetrically-generated-boolean}, every non-empty element in $\fB$ is syndetic. Let $\fC$ be a $G$-Boolean subalgebra of $2^G$ that contains $\fB$ and is maximal with respect to the property that every non-empty subset is left syndetic.  Let $X$ denote the Stone space of $\fC$.  Then by \cite{BF1997}, $X$ is a $G$-flow that is isomorphic to the universal minimal $G$-flow $\mb G$ of $G$.  Identify $e \in G$ with the principal ultrafilter it generates and let $x = e \cap \fC \in X$. Then letting $U \subseteq X$ denote the clopen set corresponding to $A$, $A = U_x$.

Conversely, assume that there is a non-empty clopen subset $U \subseteq \mb G$ and some $x \in \mb G$ such that $A = U_x$.  Let $\fB$ denote the Boolean algebra of clopen subsets of $\mb G$ and let $\fC = \{V_x : V \in \fB\} \subseteq 2^G$ be the $G$-Boolean algebra of all return sets of elements in $\fB$.  Since $\mb G$ is minimal, the map $\fB \to \fC : V \to V_x$ is an equivariant isomorphism of Boolean algebras.  So the Stone space of $\fC$ is minimal, which implies that all non-empty elements in $\fC$ are syndetic.  Since $A \in \fC$, it follows that the translation invariant Boolean subalgebra of $\fC$ generated by $A$ has the same property.  Hence by Proposition \ref{prop:symmetrically-generated-boolean}, $A$ is symmetrically syndetic.
  
The proof for the completely syndetic case is identical, except that Theorem \ref{thm:minimal-proximal-completely-syndetic} is required for the isomorphism $X \cong \gb G$. Similarly, the proof for the strongly completely syndetic case requires Theorem \ref{thm:minimal-strongly-proximal-completely-syndetic} for the isomorphism $X \cong \fb G$.
\end{proof}

We will reconsider symmetrically syndetic subsets in Section \ref{sec:dense-orbit-sets}. 

\section{Strong amenability and amenability}
\subsection{Strong amenability}

Let $G$ be a discrete group. In this section we will establish necessary and sufficient conditions, in terms of higher order syndetic subsets, for $G$ to be strongly amenable.

We will utilize the results in \cite{FTV2019}, meaning that we will need to consider subshifts of the right Bernoulli shift $2^G$ as in Section \ref{sec:subshift-dynamics}. For a finite subset $F \subseteq G \setminus \{e\}$, we will say that a subset $A \subseteq G$ is {\em $F$-avoiding} if $FA \cap A = \emptyset$. Following \cite{FTV2019}, for a finite symmetric subset $F \subseteq G$, we will say that a subshift $X \subseteq 2^G$ is an {\em $F$-witness shift} if every $B \in X$ is $F$-avoiding and for $B,C \in X$, we have $B \cap C \ne \emptyset$.

As emphasized by the authors of \cite{FTV2019}, the main technical effort in their paper is the construction of $F$-witness shifts. It follows from Proposition \ref{prop:subshift-dynamics-n-syndetic} that if $A \subseteq G$ is $2$-syndetic and $F$-avoiding, then the subshift $X \subseteq 2^G$ generated by $A$ is an $F$-witness shift. Conversely, if $X$ is an $F$-witness shift, then every $B \in X$ is $F$-avoiding from above and $2$-syndetic by Proposition \ref{prop:subshift-dynamics-n-syndetic}. Therefore, the construction of $F$-witness shifts is equivalent to the construction of $F$-avoiding $2$-syndetic subsets of $G$.

\begin{thm} \label{thm:strong-amenability}
Let $G$ be a discrete group. Then $G$ is not strongly amenable if and only if there is a proper normal subgroup $H \unlhd G$ such that for every finite subset $F \subseteq G \setminus H$, there is an $F$-avoiding completely syndetic subset $A \subseteq G$. In fact, it suffices for $A$ to be $2$-syndetic.
\end{thm}

\begin{proof}
Suppose that $G$ is not strongly amenable. Then the universal minimal proximal $G$-flow $\gb G$ is non-trivial. By \cite{FTV2019}, the kernel of the action on $\gb G$ is the FC hypercenter of $G$, which is normal in $G$. Therefore, by passing to the quotient, we can assume that the FC hypercenter is trivial. We can then invoke \cite{GTWZ2019} to see that $\gb G$ is free.

By Theorem \ref{thm:minimal-proximal-completely-syndetic}, $\gb G$ is isomorphic to the Stone space of a translation invariant Boolean subalgebra $\fB \subseteq 2^G$ consisting of completely syndetic subsets, which are in particular $2$-syndetic. Since the elements in $\fB$ correspond to a basis for the topology on $X$, it follows from freeness of $\gb G$ that for every finite subset $F \subseteq G \setminus \{e\}$, there is an $F$-avoiding subset $A \in \fB$.

Conversely, suppose that there is a proper normal subgroup $H \unlhd G$ such that for every finite subset $F \subseteq G \setminus H$, there is an $F$-avoiding $2$-syndetic subset $A \subseteq G$. By passing to the quotient, we can assume that $H = \{e\}$. Then by the remarks at the beginning of this section, $A$ generates an $F$-witness subshift $X \subseteq 2^G$. It now follows as in the last part of the proof of \cite{FTV2019}*{Theorem 1}, that $G$ is not strongly amenable.
\end{proof}

\begin{rem}
We believe there should be a more direct proof of the following fact: the group $G$ is not strongly amenable if and only there is a subset $A \subseteq G$ such that both $A$ and $A^c$ are completely syndetic. In the next section we will give a direct proof of an analogous characterization of amenability in terms of strongly completely syndetic subsets.
\end{rem}

The next result follows from Theorem \ref{thm:strong-amenability} and the main result in \cite{FTV2019}.

\begin{cor}
A discrete group $G$ is ICC if and only if for every finite subset $F \subseteq G$ there is an $F$-avoiding $2$-syndetic subset $A \subseteq G$.
\end{cor}

The next result characterizes strong amenability in terms of the existence of a single non-empty symmetrically completely syndetic subset.

\begin{thm} \label{thm:symmetric-characterization-strong-amenability}
Let $G$ be a discrete group. Then $G$ is not strongly amenable if and only if it contains a symmetrically completely syndetic subset.
\end{thm}

\begin{proof}
If $G$ is not strongly amenable, then $\gb G$ is non-trivial, and for any non-empty clopen subset $U \subseteq \gb G$ and any point $x \in \gb G$, Proposition \ref{prop:symmetric-return-set} implies that the return set $U_x \subseteq G$ is symmetrically completely syndetic. Conversely, if $A \subseteq G$ is symmetrically completely syndetic, then by Proposition \ref{prop:symmetric-return-set}, $A = U_x$ for a non-empty clopen subset $U \subseteq \gb G$ and $x \in \gb G$. In particular, $\gb G$ is non-trivial, so $G$ is not strongly amenable.
\end{proof}

\subsection{Amenability}

\begin{thm} \label{thm:amenability-test}
Let $G$ be a discrete group. Then $G$ is non-amenable if and only if there is a subset $A \subseteq G$ such that both $A$ and $A^c$ are strongly completely syndetic.
\end{thm}

\begin{proof}
Suppose that $G$ is non-amenable. Then by \cite{Fur1973}, the universal minimal strongly proximal $G$-flow $\fb G$ is non-trivial, so there is $x \in \fb G$ and $g \in G$ such that $gx \ne x$. Hence there is an open subset $U \subseteq \fb G$ with $x \in U$ such that $gU \cap U = \emptyset$. Let $A = U_x = \{h \in G : hx \in U\}$. Then $gA \cap A = \emptyset$ and by Lemma \ref{lem:minimal-strongly-proximal-strongly-syndetic}, $A$ is strongly completely syndetic. 

Conversely, suppose that there is a subset $A \subseteq G$ such that both $A$ and $A^c$ are strongly completely syndetic. Since $\fb G$ is compact and minimal, we can identify $\rC(\fb G)$ with a subalgebra of $\rC(\beta G)$. By Proposition \ref{prop:expectations} there is an idempotent equivariant unital order homomorphism $\phi : \rC(\beta G) \to \rC(\fb G)$.  There is a probability measure $\mu \in \rP(\beta G)$ such that $\phi = \phi_\mu$, where $\phi_\mu : \rC(\beta G) \to \rC(\beta G)$ is defined by $\phi_\mu(f)(g) = \langle f, g\mu \rangle$ for $f \in \rC(\beta G)$ and $g \in G$.

Arguing as in the proof of Proposition \ref{prop:minimal-strongly-proximal-strongly-syndetic}, there are $g,h \in G$ such that $(g\mu)(A) > 3/4$ and $(h\mu)(A^c) > 3/4$. Then $(h\mu)(A) \leq 1/4$. Hence letting $f = \phi_\mu(\chi_A)$ and letting $x,y \in \fb G$ be the points satisfying $\delta_x = \delta_g|_{\rC(\fb G)}$ and $\delta_y = \delta_h|_{\rC(\fb G)}$, it follows that $f(x) > 3/4$ and $f(y) < 1/4$. In particular, $\fb G$ is non-trivial. 
\end{proof}

\begin{rem}
We believe that the above proof of Theorem \ref{thm:amenability-test} is particularly interesting because it highlights the utility of the fact from Theorem \ref{thm:boundaries-injective} that $\fb G$ is maximally affinely highly proximal. However, the reverse implication of Theorem \ref{thm:amenability-test} can also be proved using F\o{}lner nets as we now show.

Recall that a group $G$ is amenable if and only if there is a F\o{}lner net for $G$. This is a net $(K_i)$ of subsets of $G$ such that for every $g \in G$, $\lim |gK_i \Delta K_i|/|K_i| = 0$.

Let $A \subseteq G$ be a subset such that both $A$ and $A^c$ are strongly completely syndetic subsets. Suppose for the sake of contradiction that $G$ is amenable and let $(K_i)$ be a F\o{}lner net for $G$. For each $i$, let $m_i$ denote the mean on $\ell^\infty(G)$ defined by $m_i(f) = \frac{1}{|K_i|} \sum_{g \in K_i} f(g)$ for $f \in \ell^\infty(G)$. Choose $\epsilon > 0$. By definition, there is a finite subset $F \subseteq G$ such that for each $i$ there is $g_i \in F$ such that $|g_iK_i \cap A| \geq (1 - \epsilon)|K_i|$. By passing to a subnet we can assume there is $g \in F$ such that $g = g_i$ for each $i$.

Let $m$ be a weak* limit point of the net $(m_i)$. Then a standard argument shows that $m$ is a left invariant mean on on $\ell^\infty(G)$. Letting $(m_j)$ be a subnet such that $\lim m_j = m$ and letting $\chi_A \in \ell^\infty(G)$ denote the indicator function corresponding to $A$,
\[
m(g^{-1}\chi_A) = \lim_j \frac{|gK_j \cap A|}{|K_j|} \geq 1 - \epsilon.
\]
Since $m$ is left invariant, this implies that $m(\chi_A) \geq 1 - \epsilon$. Since $\epsilon$ was arbitrary, it follows that $m(\chi_A) = 1$. However, since $A^c$ is also strongly completely syndetic, the same argument applied to $A^c$ implies that $m(\chi_{A^c}) = 1$. Then $1 = m(\chi_A) + m(\chi_{A^c}) > 1$,  giving a contradiction.

Note that this argument did not require the use of multisets.
\end{rem}

\begin{rem}
Note that if $A$ is completely left syndetic and there is $g \in G$ such that $gA \cap A = \emptyset$ then $gA \subseteq A^c$. Since translations and supersets of strongly completely syndetic subsets are strongly completely syndetic, it follows that $A^c$ is also completely left syndetic.
\end{rem}

\begin{example}
 Let $\bF_2 = \langle a,b \rangle$ denote the free group on two generators. Let $A \subseteq \bF_2$ denote the set of all elements in $\bF_2$ with reduced form beginning with $a$. Similarly, let $B \subseteq \bF_2$ denote the set of all elements in $\bF_2$ with reduced form beginning with $b$. Arguing as in Example \ref{ex:scs-free-group}, both $A$ and $B$ are strongly completely syndetic. Since $A \cap B = \emptyset$, Theorem \ref{thm:amenability-test} implies that $\bF_2$ is non-amenable.
\end{example}

The next result characterizes amenability in terms of the existence of a single non-empty symmetrically strongly completely syndetic subset. The proof is similar to the proof of Theorem \ref{thm:symmetric-characterization-strong-amenability}. 

\begin{thm} \label{thm:symmetric-characterization-amenability}
Let $G$ be a discrete group. Then $G$ is not amenable if and only if it contains a symmetrically strongly completely syndetic subset.
\end{thm}

\section{Dense orbit sets} \label{sec:dense-orbit-sets}

In this section we will characterize dense orbit sets in discrete groups, answering a question of Glasner, Tsankov, Weiss and Zucker from \cite{GTWZ2019}*{Question 9.6}. 

\begin{defn}
  \label{def:dense-orbit-set}
Let $G$ be a discrete group. A subset $A \subseteq G$ is a {\em dense orbit set} for $G$ if for every minimal $G$-flow $X$ and every point $x \in X$, the set $Ax$ is dense in $X$.
\end{defn}

The key observation is that the family of dense orbit sets is dual to the family of symmetrically syndetic subsets from Definition \ref{def:symmetrically-syndetic}.

\begin{thm} \label{thm:dense-orbit-set-dual}
  Let $G$ be a discrete group and $A \subseteq G$ a subset.  Then $A$ is a dense orbit set if and only if $A^c$ does not contain a symmetrically syndetic subset. Hence the dual family to the family of dense orbit sets is the family of symmetrically syndetic subsets.
\end{thm}
\begin{proof}
  By the universality of $\mb G$, if follows that $A$ is a dense orbit set if and only if $Ax \subseteq \mb G$ is dense for all $x \in \mb G$.  Since $\mb G$ is extremally disconnected (e.g. by Section \ref{sec:injectivity}), and in particular totally disconnected, the topology of $\mb G$ has a basis consisting of clopen subsets.  Hence $A$ is a dense orbit set if and only if for every $x \in \mb G$ and every non-empty clopen subset $U \subseteq \mb G$, there is $g \in A$ such that $gx \in U$.  In other words, $A$ is a dense orbit set if and only if for every $x \in \mb G$ and every non-empty clopen subset $U \subseteq \mb G$, we have $A \cap U_x \neq \emptyset$. By Proposition \ref{prop:symmetric-return-set}, sets of the form $U_x$ as above are precisely the symmetrically syndetic subsets in $G$. 
\end{proof}

\begin{rem}
We could similarly say that a subset $A \subseteq G$ is a {\em (strongly) proximally dense orbit set} if for every minimal (strongly) proximal $G$-flow $X$ and every point $x \in X$, the set $Ax$ is dense in $X$. A similar proof shows that $A$ is a (strongly) proximally dense orbit set if and only if $A^c$ does not contain a symmetrically (strongly) completely syndetic subset, and hence that the dual family to the family of (strongly) proximally dense orbit sets is the family of symmetrically (strongly) completely syndetic subsets.
\end{rem}

We will now establish several characterizations of symmetrically syndetic subsets.

\begin{lem}
  \label{lem:symmetrically-syndetic-return-set-idempotent}
  Let $G$ be a discrete group and let $A \subseteq G$ be symmetrically syndetic and let $L$ be a minimal left ideal of $\beta G$. Then there is an idempotent $u \in L$ such that $A = (\ol{A} \cap L)_u = \ol{A}_u$.
\end{lem}
\begin{proof}
 Since $A$ is symmetrically syndetic, Proposition \ref{prop:symmetric-return-set} implies that there is a clopen subset $U \subseteq \mb G$ and some $x \in \mb G$ such that $A = U_x$. We can identify $\mb G$ with $L$. By \cite{HS2012}*{Theorem 1.42}, $L$ can be written as a disjoint union $L = \bigsqcup_u uL$, where $u$ runs through all idempotents of $L$.  So in particular, there is an idempotent $u \in L$ such that $ux = x$. By Lemma \ref{lem:return-set}, $\ol{U_x} = \{y \in \beta G \mid yx \in U\}$. Therefore,
  \begin{multline*}
    \ol{A}_u
    = \{g \in G \mid gu \in \ol{A}\}
    = \{g \in G \mid gu \in \ol{U_x}\} \\
    = \{g \in G \mid gux \in U\} 
    = U_{ux}
    = U_x
    = A
  \end{multline*}
Since $L$ is a left ideal, it is clear that $(\ol{A} \cap L)_u = \ol{A}_u$.
\end{proof}

The next result provides a simpler characterization of symmetrically syndetic subsets, and reveals a connection to the notion of a central subset. Central sets in $\bN$ were introduced by Furstenberg \cite{Fur1981}.  Later, central sets for general semigroups were introduced by Hindman and Maleki \cite{HM1996}. Both algebraic and topological dynamical characterizations of central sets are now known. We will use the fact that for a discrete group $G$, a subset $A \subseteq G$ is {\em central} if and only if the closure $\ol{A} \subseteq \beta G$ contains a minimal idempotent of $\beta G$ (see e.g. \cite{HS2012}). 

\begin{prop}
  \label{prop:symmetrically-syndetic-combintorial-characterization}
  Let $G$ be a discrete group. A subset $A \subseteq G$ is symmetrically syndetic if and only if for every pair of finite subsets $F_1 \subseteq A$ and $F_2 \subseteq A^c$, the set $(\cap_{f_1 \in F_1} f_1^{-1} A) \cap (\cap_{f_2 \in F_2} f_2^{-1} A^c)$ is syndetic.  In this case, the latter subset is central
\end{prop}
\begin{proof}
  Suppose that $A$ is symmetrically syndetic.  Identify $\mb G$ with a fixed minimal left ideal $L$ of $\beta G$.  Then by Proposition \ref{lem:symmetrically-syndetic-return-set-idempotent}, there is an idempotent $u \in L$ such that $A = \ol{A}_u$.  Therefore, $u \in a^{-1} \ol{A}$ for every $a \in A$ and $u \notin b^{-1} \ol{A}$ for every $b \in A^c$. The latter statement is equivalent to $u \in b^{-1}\ol{A}^c$ for every $b \in A^c$. Therefore, for every pair of finite subsets $F_1 \subseteq A$ and $F_2 \subseteq A^c$,
  \[
    u \in (\bigcap_{f_1 \in F_1} f_1^{-1} \ol{A}) \cap (\bigcap_{f_2 \in F_2} f_2^{-1} \ol{A}^c) = \ol{(\bigcap_{f_1 \in F_1} f_1^{-1} A) \cap (\bigcap_{f_2 \in F_2} f_2^{-1} A^c)}.
  \]
This implies that the intersection $(\cap_{f_1 \in F_1} f_1^{-1} A) \cap (\cap_{f_2 \in F_2} f_2^{-1} A^c)$ is non-empty, so by the definition of a symmetrically syndetic subset, it is syndetic. We further observe that is central.

  Conversely, suppose that $A$ is such that for every pair of finite subsets $F_1 \subseteq A$ and $F_2 \subseteq A^c$, the set $(\cap_{f_1 \in F_1} f_1^{-1} A) \cap (\cap_{f_2 \in F_2} f_2^{-1} A^c)$ is syndetic.  Fix a minimal left ideal $L$ of $\beta G$. Since the sets $(\cap_{f_1 \in F_1} f_1^{-1} A) \cap (\cap_{f_2 \in F_2} f_2^{-1} A^c)$ are syndetic for any choice of finite subsets $F_1 \subseteq A$, $F_2 \subseteq A^c$, the family of sets $\{a^{-1} \ol{A} \cap L : a \in A\} \cup \{b^{-1} \ol{A}^c \cap L : b \in A^c\}$ has the finite intersection property. By the compactness of $L$, there is $x \in (\cap_{a \in A} a^{-1} \ol{A}) \cap (\cap_{b \in A^c} b^{-1}\ol{A}^c) \cap L$. Letting $U = \ol{A} \cap L$, $U$ is a non-empty clopen subset and $ax \in U$ for all $a \in A$ and $bx \notin U$ for all $b \in A^c$. In other words, $A = U_x$. Hence by Proposition \ref{prop:symmetrically-syndetic-dynamical-characterization}, $A$ is symmetrically syndetic.
\end{proof}

We will now give another characterization of symmetrically syndetic subsets that uses the finer structure of the semigroup $\beta G$. It extends ideas from the proof of Proposition \ref{prop:symmetrically-syndetic-combintorial-characterization}. Recall that an idempotent in $\beta G$ is minimal if it belongs to a minimal left ideal of $\beta G$. We begin by establishing some facts about minimal idempotents in $\beta G$.
\begin{defn}
  Let $G$ be a discrete group. We will say that a family $E \subseteq \beta G$ of minimal idempotents is {\em compatible} if 
  \begin{enumerate}
  \item for every minimal left ideal $L$ of $\beta G$ there is a unique $u \in E \cap L$,
  \item $uv = v$ for all $u,v \in E$.
  \end{enumerate}
\end{defn}

\begin{rem}
  For minimal left ideals $L_1,L_2$ of $\beta G$ with corresponding idempotents $u_1,u_2 \in E$, the maps $L_1 \to L_2 : x \to xu_2$ and $L_2 \to L_1 : y \to yu_1$ are isomorphisms of $G$-flows (see e.g. \cite{Gla1976}*{Proposition I.2.5}).
\end{rem}

\begin{lem} \label{lem:existence-compatible-family}
  Let $G$ be a discrete group, let $L$ be a minimal left ideal of $\beta G$ and let $u \in L$ be an idempotent. Then there is a unique compatible family of minimal idempotents for $\beta G$ containing $u$. In particular, there is a compatible family of minimal idempotents in $\beta G$.
\end{lem}

\begin{proof}
Let $\{L_i\}_i$ be the family of minimal left ideals of $\beta G$. The structure theory of compact right topological semigroups (see e.g. \cite{HS2012}*{Theorem 1.64}) implies that for each $i$ there is a unique idempotent $v_i \in L_i$ such that $u v_i = v_i$ and $v_i u = u$. Then for each $j$, it follows that $v_i v_j = v_i (u v_j) = (v_i u) v_j = u v_j = v_j$. Therefore, the family $E = \{v_i\}_i$ is compatible. If $L = L_i$, then $u^2 = u$, so that the above uniqueness of the idempotent $v_i$ implies $u = v_i$.
\end{proof}

The next result provides further motivation for the name ``symmetrically syndetic'' by showing that a subset is symmetrically syndetic if and only if it is invariant under taking return sets relative to a compatible family of minimal idempotents.

\begin{prop} \label{prop:symmetrically-syndetic-iff-idempotents}
Let $G$ be a discrete group. A subset $A \subseteq G$ is symmetrically syndetic if and only if there is a compatible family of minimal idempotents $E \subseteq \beta G$ such that $A = \ol{A}_u$ for all $u \in E$.
\end{prop}

\begin{proof}
  Suppose that $A$ is symmetrically syndetic and fix a minimal left ideal $L \subseteq \beta G$.  By Proposition \ref{prop:symmetrically-syndetic-dynamical-characterization}, there is an idempotent $u \in L$ such that $A = \ol{A}_u$. By Lemma \ref{lem:existence-compatible-family} there is a compatible family of minimal idempotents $E \subseteq \beta G$ such that $u \in E$.  By Lemma \ref{lem:return-set}, we have $\ol{A} = \ol{\ol{A}_u} = \{x \in \beta G \mid xu \in \ol{A} \}$.  Hence for all $v \in E$, we find that
  \[
    \ol{A}_v = \{g \in G \mid gvu \in \ol{A} \} = \{g \in G \mid gu \in \ol{A} \} = \ol{A}_u = A.
  \]
  
  Conversely, suppose there is a compatible family of minimal idempotents $E \subseteq \beta G$ such that $\ol{A}_u = A$ for all $u \in E$. Let $L$ be a minimal left ideal of $\beta G$ and let $u \in E$ be the unique idempotent such that $u \in L$. Let $U = \ol{A} \cap L$. Then $U$ is clopen and $g \in U_u$ if and only if $gu \in U$ if and only if $gu \in \ol{A}$. Hence $A = \ol{A}_u = U_u$, so $A$ is symmetrically syndetic.
\end{proof}

\bibliographystyle{alpha}
\bibliography{strong-amenability}

\begin{bibdiv}
\begin{biblist}

\bib{Arv1976}{book}{
      author={Arveson, William},
       title={An invitation to {$C\sp*$}-algebras},
      series={Graduate Texts in Mathematics},
   publisher={Springer-Verlag, New York-Heidelberg},
        date={1976},
      volume={39},
}

\bib{AG1977}{article}{
      author={Auslander, J.},
      author={Glasner, S.},
       title={Distal and highly proximal extensions of minimal flows},
        date={1977},
        ISSN={0022-2518},
     journal={Indiana Univ. Math. J.},
      volume={26},
      number={4},
       pages={731\ndash 749},
  url={https://doi-org.proxy.lib.uwaterloo.ca/10.1512/iumj.1977.26.26057},
}

\bib{BF1997}{article}{
      author={Balcar, Bohuslav},
      author={Franek, Frantisek},
       title={Structural properties of universal minimal dynamical systems for
  discrete semigroups},
        date={1997},
        ISSN={0002-9947},
     journal={Trans. Amer. Math. Soc.},
      volume={349},
      number={5},
       pages={1697\ndash 1724},
  url={https://doi-org.proxy.lib.uwaterloo.ca/10.1090/S0002-9947-97-01868-0},
}

\bib{BHM1998}{inproceedings}{
      author={Bergelson, Vitaly},
      author={Hindman, Neil},
      author={McCutcheon, Randall},
       title={Notions of size and combinatorial properties of quotient sets in
  semigroups},
        date={1998},
   booktitle={Proceedings of the 1998 {T}opology and {D}ynamics {C}onference
  ({F}airfax, {VA})},
      volume={23},
       pages={23\ndash 60},
}

\bib{BHS2012}{article}{
      author={Bergelson, Vitaly},
      author={Hindman, Neil},
      author={Strauss, Dona},
       title={Strongly central sets and sets of polynomial returns mod 1},
        date={2012},
        ISSN={0002-9939},
     journal={Proc. Amer. Math. Soc.},
      volume={140},
      number={8},
       pages={2671\ndash 2686},
  url={https://doi-org.proxy.lib.uwaterloo.ca/10.1090/S0002-9939-2011-11129-8},
}

\bib{Ell1960}{article}{
      author={Ellis, Robert},
       title={Universal minimal sets},
        date={1960},
        ISSN={0002-9939},
     journal={Proc. Amer. Math. Soc.},
      volume={11},
       pages={540\ndash 543},
         url={https://doi-org.proxy.lib.uwaterloo.ca/10.2307/2034707},
      review={\MR{117716}},
}

\bib{EK2019}{article}{
      author={Erschler, Anna},
      author={Kaimanovich, Vadim},
       title={{Arboreal structures on groups and the associated boundaries}},
        date={2019},
     journal={arXiv preprint arXiv:1903.02095},
}

\bib{frischhartmantamuzferdowsi18}{article}{
      author={Frisch, Joshua},
      author={Hartman, Yair},
      author={Tamuz, Omer},
      author={Ferdowsi, Pooya~Vahidi},
       title={{Choquet-Deny groups and the infinite conjugacy class
  property.}},
        date={2019},
     journal={Ann. Math.},
      volume={190},
      number={1},
       pages={307\ndash 320},
}

\bib{FTV2019}{article}{
      author={Frisch, Joshua},
      author={Tamuz, Omer},
      author={Vahidi~Ferdowsi, Pooya},
       title={Strong amenability and the infinite conjugacy class property},
        date={2019-12-01},
        ISSN={1432-1297},
     journal={Invent. Math.},
      volume={218},
      number={3},
       pages={833\ndash 851},
         url={https://doi.org/10.1007/s00222-019-00896-z},
}

\bib{Fur1973}{inproceedings}{
      author={Furstenberg, Harry},
       title={Boundary theory and stochastic processes on homogeneous spaces},
        date={1973},
   booktitle={Harmonic analysis on homogeneous spaces ({P}roc. {S}ympos. {P}ure
  {M}ath., {V}ol. {XXVI}, {W}illiams {C}oll., {W}illiamstown, {M}ass., 1972)},
       pages={193\ndash 229},
}

\bib{Fur1981}{book}{
      author={Furstenberg, Harry},
       title={Recurrence in ergodic theory and combinatorial number theory},
      series={M. B. Porter Lectures},
   publisher={Princeton University Press, Princeton, N.J.},
        date={1981},
      volume={10},
        ISBN={0-691-08269-3},
}

\bib{GTWZ2019}{article}{
      author={Glasner, Eli},
      author={Tsankov, Todor},
      author={Weiss, Benjamin},
      author={Zucker, Andy},
       title={Bernoulli disjointness},
        date={2020},
     journal={to appear in Duke Math. J.},
}

\bib{Gla1975}{article}{
      author={Glasner, Shmuel},
       title={Relatively invariant measures},
        date={1975},
        ISSN={0030-8730},
     journal={Pacific J. Math.},
      volume={58},
      number={2},
       pages={393\ndash 410},
  url={http://projecteuclid.org.proxy.lib.uwaterloo.ca/euclid.pjm/1102905671},
      review={\MR{380756}},
}

\bib{Gla1976}{book}{
      author={Glasner, Shmuel},
       title={Proximal flows},
      series={Lecture Notes in Mathematics},
   publisher={Springer-Verlag, Berlin-New York},
        date={1976},
      volume={517},
}

\bib{Gle1958}{article}{
      author={Gleason, Andrew~M.},
       title={Projective topological spaces},
        date={1958},
        ISSN={0019-2082},
     journal={Illinois J. Math.},
      volume={2},
       pages={482\ndash 489},
  url={http://projecteuclid.org.proxy.lib.uwaterloo.ca/euclid.ijm/1255454110},
}

\bib{Gon1968}{article}{
      author={Gonshor, Harry},
       title={Injective hulls of {$C\sp*$} algebras},
        date={1968},
        ISSN={0002-9947},
     journal={Trans. Amer. Math. Soc.},
      volume={131},
       pages={315\ndash 322},
         url={https://doi-org.proxy.lib.uwaterloo.ca/10.2307/1994948},
      review={\MR{229053}},
}

\bib{Gon1970}{article}{
      author={Gonshor, Harry},
       title={Injective hulls of {$C\sp*$} algebras. {II}},
        date={1970},
        ISSN={0002-9939},
     journal={Proc. Amer. Math. Soc.},
      volume={24},
       pages={486\ndash 491},
         url={https://doi-org.proxy.lib.uwaterloo.ca/10.2307/2037393},
      review={\MR{287318}},
}

\bib{HP2011}{article}{
      author={Hadwin, Don},
      author={Paulsen, Vern~I.},
       title={Injectivity and projectivity in analysis and topology},
        date={2011},
        ISSN={1674-7283},
     journal={Sci. China Math.},
      volume={54},
      number={11},
       pages={2347\ndash 2359},
  url={https://doi-org.proxy.lib.uwaterloo.ca/10.1007/s11425-011-4285-7},
      review={\MR{2859698}},
}

\bib{Ham1985}{article}{
      author={Hamana, Masamichi},
       title={Injective envelopes of {$C^\ast$}-dynamical systems},
        date={1985},
        ISSN={0040-8735},
     journal={Tohoku Math. J. (2)},
      volume={37},
      number={4},
       pages={463\ndash 487},
         url={https://doi-org.proxy.lib.uwaterloo.ca/10.2748/tmj/1178228589},
}

\bib{HM1996}{article}{
      author={Hindman, Neil},
      author={Maleki, Amir},
      author={Strauss, Dona},
       title={Central sets and their combinatorial characterization},
        date={1996},
        ISSN={0097-3165},
     journal={J. Combin. Theory Ser. A},
      volume={74},
      number={2},
       pages={188\ndash 208},
         url={https://doi-org.proxy.lib.uwaterloo.ca/10.1006/jcta.1996.0048},
}

\bib{HS2012}{book}{
      author={Hindman, Neil},
      author={Strauss, Dona},
       title={Algebra in the {S}tone-\v{C}ech compactification},
      series={De Gruyter Textbook},
   publisher={Walter de Gruyter \& Co., Berlin},
        date={2012},
        ISBN={978-3-11-025623-9},
        note={Second revised and extended edition},
}

\bib{Kad1951}{article}{
      author={Kadison, Richard~V.},
       title={A representation theory for commutative topological algebra},
        date={1951},
        ISSN={0065-9266},
     journal={Mem. Amer. Math. Soc.},
      volume={7},
       pages={39},
      review={\MR{44040}},
}

\bib{KK2017}{article}{
      author={Kalantar, Mehrdad},
      author={Kennedy, Matthew},
       title={Boundaries of reduced {$C^*$}-algebras of discrete groups},
        date={2017},
        ISSN={0075-4102},
     journal={J. Reine Angew. Math.},
      volume={727},
       pages={247\ndash 267},
         url={https://doi-org.proxy.lib.uwaterloo.ca/10.1515/crelle-2014-0111},
}

\bib{Ken2018}{article}{
      author={Kennedy, Matthew},
       title={An intrinsic characterization of {$C^*$}-simplicity},
        date={2020},
        ISSN={0012-9593},
     journal={Ann. Sci. \'{E}c. Norm. Sup\'{e}r. (4)},
      volume={53},
      number={5},
       pages={1105\ndash 1122},
      review={\MR{4174855}},
}

\bib{KS2019}{article}{
      author={Kennedy, Matthew},
      author={Schafhauser, Christopher},
       title={Noncommutative boundaries and the ideal structure of reduced
  crossed products},
        date={2019},
        ISSN={0075-4102},
     journal={Duke Math. J.},
      volume={168},
      number={17},
       pages={3215\ndash 3260},
         url={https://doi.org/10.1215/00127094-2019-0032},
}

\bib{PT2009}{article}{
      author={Paulsen, Vern~I.},
      author={Tomforde, Mark},
       title={Vector spaces with an order unit},
        date={2009},
        ISSN={0022-2518},
     journal={Indiana Univ. Math. J.},
      volume={58},
      number={3},
       pages={1319\ndash 1359},
  url={https://doi-org.proxy.lib.uwaterloo.ca/10.1512/iumj.2009.58.3518},
}

\bib{Zuc2020}{article}{
      author={Zucker, Andy},
       title={Maximally highly proximal flows},
        date={2020},
     journal={to appear in Ergod. Theory Dyn. Syst.},
}

\end{biblist}
\end{bibdiv}

\end{document}